\documentclass[pra,twocolumn,showpacs,superscriptaddress,floatfix,nofootinbib]{revtex4-2}
\usepackage{graphicx} % Required for inserting images
\usepackage[english,russian]{babel}
\usepackage{amsmath,amssymb,mathrsfs,esint,amsthm}
\usepackage{tabularx}
\usepackage[caption=false]{subfig}
\usepackage{multirow}
\usepackage{algorithm}
\usepackage[noend]{algpseudocode}
\usepackage[bookmarks=true,
   colorlinks=true,
   linkcolor=blue,
   urlcolor=blue,
   citecolor=blue,
   bookmarks=true,
   hyperindex=true
]{hyperref}
\usepackage{comment}
\usepackage{bm}
\usepackage[normalem]{ulem}
\usepackage{svg}
\usepackage{xifthen}
\usepackage{physics}

\newtheorem{theorem}{Утверждение}
\newtheorem{example}{Пример}[section]
\newtheorem{definition}{Определение}

\begin{document}

\title{Transforming optimization problems into a QUBO form: A tutorial}

\author{Alexander M. Semenov}
\affiliation{Russian Quantum Center, Skolkovo, Moscow 121205, Russia}
\affiliation{Cloud Quantum Technologies,  Moscow 123112, Russia}
\author{Sergey R. Usmanov}
\affiliation{Russian Quantum Center, Skolkovo, Moscow 121205, Russia}
\affiliation{Cloud Quantum Technologies,  Moscow 123112, Russia}
\author{Aleksey K. Fedorov}
\affiliation{Russian Quantum Center, Skolkovo, Moscow 121205, Russia}
\affiliation{Cloud Quantum Technologies,  Moscow 123112, Russia}

\begin{abstract}
Practically relevant problems of quadratic optimization often contain multidimensional arrays of variables interconnected by linear constraints, such as equalities and inequalities. 
The values of each variable depend on its specific meaning and can be binary, integer, discrete, and continuous. 
These circumstances make it technically difficult to reduce the original problem statement to the QUBO form.
The paper identifies and considers three main transformations of the original problem statement, namely, the transition from a multidimensional to a one-dimensional array of variables, 
the transition in mixed problems to binary variables, and the inclusion of linear constraints in the objective function in the form of quadratic penalties.
Convenient formulas for calculations are presented and proven, simplifying the implementation of these transformations.
In particular, the formulas for the transition in the problem statement from a multidimensional to a one-dimensional array of variables are based on the use of the Kronecker product of matrices.
The considered transformations are illustrated by numerous examples.
\end{abstract}

\maketitle

\begin{widetext}

\section{Введение}
Сложность многих задач оптимизации, важных с практической точки зрения,  например, задач составления расписаний производственных процессов и маршрутов, является экспоненциальной по времени. 
Современные компьютеры не в состоянии за разумное время найти точное решение таких задач для практически релевантных размерностей.
Одним из возможных путей решения данной проблемы является создание машин, основанных на различных физических принципах~\cite{4}, специально предназначенных для решения задач оптимизации. 
В качестве примера одной из таких машин можно привести оптическую когерентную машину Изинга~\cite{5}. 
Для использования специализированных машин или программных симуляторов их работы задача оптимизации должна быть представлена либо в форме задачи квадратичной бинарной оптимизации без ограничений (QUBO),
\begin{equation}
	\label{nameIfNeededRefInText}
	\frac{1}{2}x^TQx+v^Tx \to \min,\quad Q^T=Q,\quad x_i \in \{0,1\},
\end{equation}
либо в эквивалентной ей форме -- в форме задачи Изинга
\begin{equation}
	-\frac{1}{2}s^TJs-h^Ts \to \min, \quad s_i \in \{-1,1\}.
\end{equation}
Задача Изинга получается после замены переменных \(x=2s-1\) в задаче QUBO. 
Матрицы квадратичных форм эквивалентных задач связаны следующим образом
\begin{equation}
	J_{ij}=-\frac14(1-\delta_{ij})q_{ij}\quad \textrm{и}\quad h_i=-\frac{v_i}{2}-\frac{1}{4}\sum_{j=1}^{n}q_{ij}.
\end{equation}
В случае задачи в форме QUBO, используя равенства $x_i^2=x_i$, линейную часть целевой функции можно внести на диагональ матрицы квадратичной формы. 

Формулировка практических задач оптимизации в форме QUBO и Изинга давно привлекает внимание многих специалистов в различных областях знаний: 
квантовые вычисления~\cite{6,7,8}, 
комбинаторная оптимизация \cite{10,11,13,14}, 
машинное обучение~\cite{15,18,19}, телекоммуникации~\cite{9,17,22}, финансы~\cite{16,20,21} и др. 
На сайте \href{https://blog.xa0.de/post/List-of-QUBO-formulations/}{List of QUBO formulations} 
приведен список из $112$ практически важных задач, сведенных к форме QUBO, и даны ссылки на соответствующие статьи специалистов.

Технические приемы сведения различных задач оптимизации к форме QUBO и Изинга в целом известны~\cite{1,2,3,12}. 
Результатом приведения задачи к форме QUBO или Изинга, как правило, является квадратичная форма от переменных с несколькими индексами. 
Авторы предполагают, что выбор способа упорядочивания переменных и формирования матрицы квадратичной формы будет сделан на этапе программной реализации алгоритма решения задачи. 
С другой стороны, итерационные алгоритмы решения смешанных задач часто требуют формирования матрицы квадратичной формы на каждой итерации. 
Простота и вычислительная эффективность данного действия становится актуальной.   

Основная цель данной работы -- рассмотреть и развить общие технические приемы сведения исходной задачи квадратичной оптимизации к форме QUBO.
Под результатом сведения исходной задачи к форме QUBO мы будем понимать  формулы для вычисления матрицы квадратичной формы и коэффициентов линейной части. 

При сведении практической задачи к форме QUBO можно выделить три основных преобразования постановки задачи: переход от многомерного к одномерному массиву переменных (векторизация); 
переход от целых, дискретных и непрерывных переменных к бинарным (бинаризация) и включение линейных ограничений в целевую функцию в виде квадратичных штрафов.
Преобразование векторизации заключается в упорядочивании многомерных переменных и представлении задачи в форме стандартной задачи квадратичной оптимизации с линейными ограничениями. 
Преобразование бинаризации заключается в построении аффинного отображения из множества 
всех бинарных векторов фиксированной длины в декартово произведение множеств значений исходных переменных и представлении задачи в форме квадратичной бинарной оптимизации с линейными ограничениями.
Преобразование включения линейных ограничений в целевую функцию в виде квадратичных штрафов заключается в представлении задачи в форме QUBO и указании условий на штрафные коэффициенты, 
при которых полученная задача эквивалентна исходной.
Каждое из данных преобразований подробно описано в разделах \ref{p2} -- \ref{p4} и проиллюстрировано многочисленными примерами. 
В параграфе \ref{p22} доказаны утверждения, позволяющие эффективно использовать произведение Кронекера матриц при векторизации практических задач оптимизации. 
В параграфе \ref{p42} получены условия эквивалентности исходной задачи и задачи в форме QUBO.

\section{Преобразование многомерных массивов переменных}\label{p2}

\subsection{Упорядочивание наборов индексов}
В случае переменных с двумя и более индексами квадратичная целевая функция записывается в виде
\begin{equation}
	\sum_{i_1\ldots j_1,i_2\ldots j_2}q_{i_1\ldots j_1,i_2\ldots j_2}x_{i_1\ldots j_1}x_{i_2\ldots j_2}.
\end{equation}
Для записи данной суммы в матричном виде необходимо упорядочить множество наборов индексов. 
Имеется два стандартных способа упорядочивания наборов индексов: в стиле языка C и в стиле языка Fortran. 
В языке C упорядочивание наборов индексов предполагает самое быстрое изменение последнего индекса и самое медленное изменение первого индекса в наборе, а в языке Fortran -- наоборот.

Например, если индекс $i$ в наборе $ij$ принимает значения от 1 до 2, а индекс $j$ -- значения от 1 до 3, то упорядочивание в стиле C и в стиле Fortran имеют вид
$\begin{bmatrix}
1 & 2 & 3\\
4 & 5 & 6
\end{bmatrix}$
и
$\begin{bmatrix}
1 & 3 & 5\\
2 & 4 & 6
\end{bmatrix}$ соответственно. 
В общем случае, если первый индекс в наборе изменяется от 1 до $n$, а второй от 1 до $m$, то данные упорядочивания записываются в виде
\( N(i,j)=m(i-1)+j\) и \(N(i,j)=n(j-1)+i\) соответственно. 
Далее мы будем использовать упорядочивание наборов индексов в стиле языка Fortran. 
Например, для переменных с тремя индексами будем использовать следующее упорядочивание \(N(i,j,k)=nm(k-1)+n(j-1)+i\). 

Упорядоченный в стиле языка Fortran массив переменных $x$, записанный в виде вектора-столбца, будем обозначать через $\bar{x}$.

\subsection{Использование произведения Кронекера матриц при переходе от многомерного к одномерному массиву переменных}\label{p22}

В данном параграфе мы будем часто использовать произведение Кронекера матриц. Напомним определение и основные свойства произведения Кронекера.
\begin{definition}
Произведение Кронекера -- бинарная операция над матрицами произвольного размера, обозначается $\otimes$.
Если $A$ -- матрица размера $m\times n$, $B$ — матрица размера $p\times q$, тогда произведение Кронекера есть блочная матрица размера $mp\times nq$
\begin{equation}
    A\otimes B=
    \begin{pmatrix}
    a_{11}B &\ldots & a_{1n}B\\
    \vdots  & \ddots & \vdots \\
    a_{m1}B & \ldots & a_{mn}B
    \end{pmatrix}.
\end{equation}
Элемент произведения Кронекера \(a_{ij}b_{kl}\) находится в строке с номером \(p(i-1)+k\) и в столбце с номером \(q(j-1)+l\).
\end{definition}
Произведение Кронекера является билинейной ассоциативной операцией и хорошо согласовано с основными операциями над матрицами:
\begin{equation}
    (A\otimes B)^T=A^T\otimes B^T
\end{equation}
\begin{equation}
    (A\otimes B)(C\otimes D)=AC\otimes BD
\end{equation}
\begin{equation}
    (A\otimes B)^{-1}=A^{-1}\otimes B^{-1}
\end{equation}
Приведенные равенства выполняются, если все операции над матрицами определены.

Произведение Кронекера удобно использовать при векторизации задачи.
Следующее утверждение позволяет компактно записать целевую функцию после упорядочивания массива переменных.
\begin{theorem}\label{S1}
Пусть $A,B$ и $C$ -- симметричные матрицы размера $n,m$ и $p$, соответственно. Тогда
\begin{equation}
\sum_{i_1,j_1,i_2,j_2}a_{i_1i_2}b_{j_1j_2}x_{i_1j_1}x_{i_2j_2}=\bar{x}^T(B\otimes A)\bar{x}
\end{equation}
и
\begin{equation}
\sum_{i_1,j_1,k_1,i_2,j_2,k_2}a_{i_1i_2}b_{j_1j_2}c_{k_1k_2}x_{i_1j_1k_1}x_{i_2j_2k_2}=
\bar{x}^T(C\otimes B \otimes A)\bar{x}.
\end{equation}
\end{theorem}
\begin{proof}
Заметим, что аналогичное утверждение можно сформулировать для массивов переменных
любой размерности. В частности, для одномерного массива
переменных утверждение очевидно
\begin{equation}
    \sum_{i_1,i_2}a_{i_1i_2}x_{i_1}x_{i_2}=x^TAx=\bar{x}^TA\bar{x}.
\end{equation}
Докажем утверждение для массива переменных размерности 2. Для размерностей массива переменных больше 2 доказательство проводится аналогично.
Пусть $E_{kl}$ матрица, у которой только один элемент $e_{kl}=1$, а остальные элементы равны нулю. Размер матрицы определяется диапазонами изменения индексов. В силу определения произведения Кронекера, его билинейности и выбранного способа упорядочивания наборов индексов, получим
\begin{multline}
\bar{x}^T(B\otimes A)\bar{x}=\bar{x}^T\Bigl(\sum_{j_1,j_2}b_{j_1j_2}E_{j_1j_2}\otimes \sum_{i_1,i_2}a_{i_1i_2}E_{i_1i_2}\Bigr )\bar{x} = \ldots \\
= \sum_{i_1,j_1,i_2,j_2}a_{i_1i_2}b_{j_1j_2}\bar{x}^T(E_{j_1j_2}\otimes E_{i_1i_2})\bar{x}=\ldots \\
= \sum_{i_1,j_1,i_2,j_2}a_{i_1i_2}b_{j_1j_2}\bar{x}_{n(j_1-1)+i_1}\bar{x}_{n(j_2-1)+i_2}=\ldots \\
= \sum_{i_1,j_1,i_2,j_2}a_{i_1i_2}b_{j_1j_2}x_{i_1j_1}x_{i_2j_2}.
\end{multline}
Что и требовалось доказать.
\end{proof}
Следующее утверждение позволяет компактно записать линейные ограничения после упорядочивания массива переменных.
\begin{theorem}\label{S2}
Пусть $A,B$ и $C$ -- матрицы размера $l_1\times n,l_2\times m$ и $l_3\times p$, соответственно. Если линейные ограничения задачи с двумерным массивом переменных записываются в виде
\begin{equation}
    \sum_{i,j}a_{ri}b_{sj}x_{ij} = d_{rs} ,\quad r=1,\ldots, l_1,\quad s=1, \ldots, l_2,
\end{equation}
то данный массив ограничений, упорядоченный в стиле языка Fortran, имеет вид
\begin{equation}
    (B\otimes A)\bar{x} = \bar{d}
\end{equation}
и
\begin{equation}
    \|(B\otimes A)\bar{x} - \bar{d}\|^2=\bar{x}^T(B^TB\otimes A^TA)\bar{x}-2((B^T\otimes A^T)\bar{d})^T\bar{x}+\|\bar{d}\|^2.
\end{equation}
Аналогично, если линейные ограничения задачи с трехмерным массивом переменных
записываются в виде
\begin{equation}
    \sum_{i,j,k}a_{ri}b_{sj}c_{tk}x_{ijk} = d_{rst} ,\quad r=1,\ldots,l_1,\quad s=1, \ldots, l_2,\quad 
t=1,\ldots ,l_3,
\end{equation}
то данный массив ограничений, упорядоченный в стиле языка Fortran, имеет вид
\begin{equation}
    (C\otimes B\otimes A)\bar{x} = \bar{d}
\end{equation}
и
\begin{multline}
    \|(C\otimes B\otimes A)\bar{x} - \bar{d}\|^2= \ldots \\ =\bar{x}^T(C^TC\otimes B^TB\otimes A^TA)\bar{x}
-2((C^T\otimes B^T\otimes A^T)\bar d)^T\bar{x}+\|\bar{d}\|^2.
\end{multline}
\end{theorem}
\begin{proof}
Заметим, что аналогичное утверждение можно сформулировать для массивов переменных
любой размерности. В частности, для одномерного массива, $x=\bar{x}$ и $d=\bar{d}$. Система уравнений
\begin{equation}
    \sum_{i}a_{ri}x_{i} = d_{r} , r=1,\ldots, l_1
\end{equation}
в матричной записи имеет нужный нам вид $A\bar{x} = \bar{d}$
и квадрат нормы разности левой и правой частей системы равен
\begin{multline}
\|A\bar{x} - \bar{d}\|^2=(A\bar{x} - \bar{d})^T(A\bar{x} - \bar{d})=\ldots \\
=(\bar{x}^TA^T - \bar{d}^T)(A\bar{x} - \bar{d})=\bar{x}^T(A^TA)\bar{x}-2(A^T\bar d)^T\bar{x}+\|\bar{d}\|^2.
\end{multline}
Докажем утверждение для массива переменных размерности 2. Для размерностей массива переменных больше 2 доказательство проводится аналогично.
Пусть $E_{kl}$ матрица, у которой только один элемент $e_{kl}=1$, а остальные элементы равны нулю. Размер матрицы определяется диапазонами изменения индексов. Обозначим через $e_k$ вектор-столбец размера $l_1l_2$, у которого  $k$-ая координата равна единице, а остальные координаты равны нулю.  В силу определения произведения Кронекера, его билинейности и выбранного способа упорядочивания наборов индексов, получим
\begin{multline}
(B\otimes A)\bar{x}=\Bigl (\sum_{r,i}a_{ri}E_{ri}\otimes \sum_{s,j}b_{sj}E_{sj}\Bigr )= \ldots \\
= \sum_{r,s}\sum_{i,j}a_{ri}b_{sj}(E_{sj}\otimes E_{ri})\bar{x}=\ldots \\
=\sum_{r,s}\Bigl (\sum_{i,j}a_{ri}b_{sj}\bar{x}_{m(j-1)+i}\Bigr )e_{l_2(s-1)+r}=\ldots \\
=\sum_{r,s}\Bigl (\sum_{i,j}a_{ri}b_{sj}x_{ij}\Bigr )e_{l_2(s-1)+r}.
\end{multline}
Из последнего равенства следует, что исходная система уравнений, упорядоченная в стиле языка Fortran, совпадает с системой уравнений
\((B\otimes A)\bar{x} = \bar{d}.\)

Используя свойства произведения Кронекера, вычислим квадрат нормы разности
\begin{multline}
\|(B\otimes A)\bar{x}-\bar{d}\|^2=((B\otimes A)\bar{x}-\bar{d})^T((B\otimes A)\bar{x}-\bar{d})=\ldots \\
=(\bar{x}^T(B^T\otimes A^T)-\bar{d}^T)((B\otimes A)\bar{x}-\bar{d})=\ldots \\
=\bar{x}^T(B^TB\otimes A^TA)\bar{x}-2((B^T\otimes A^T)\bar d)^T\bar x +\|\bar{d}\|^2.
\end{multline}
Что и требовалось доказать.
\end{proof}

\subsection{Примеры преобразований}
Обозначим через $E_n$  единичную матрицу размера $n$, через $I_{m\times n}$ матрицу размера $m\times n$, все элементы которой
равны единице, и через $\delta_{ij}$ символ Кронекера.

\begin{example}
Возведение линейных комбинаций переменных в квадрат.

Пусть $a$ -- вектор-строка размера $n$ и $b$ -- вектор-строка размера $m$. Тогда, по утверждению \ref{S2} и
свойствам произведения Кронекера
\begin{multline}
\Bigl( \sum_{i,j}a_ib_jx_{ij}\Bigr)^2=((b\otimes a)\bar{x})^T((b\otimes a)\bar{x})=\ldots \\
=\bar{x}^T(b^T\otimes a^T)(b \otimes a)\bar{x}=\bar{x}^T(b^Tb\otimes a^Ta)\bar{x}.
\end{multline}
\end{example}
\begin{example}\label{ex22}

Суммирование по индексам. 
\begin{enumerate}
    \item Пусть левая часть системы уравнений имеет вид суммирования по одному индексу
    \begin{equation}
        \sum_{i=1}^nx_{ij} = b_j,\quad j=1,\ldots ,m.
    \end{equation}
    Заметим, что
    \begin{equation}
        \sum_{i,j}\delta_{sj}x_{ij}=\sum_ix_{is},\quad s=1,\ldots,m.
    \end{equation}
    Тогда, по утверждению \ref{S2}, система данных уравнений после векторизации 
    задачи имеет вид
    \begin{equation}
        (E_m\otimes I_{1\times n})\bar{x} = b
    \end{equation}

    \item Пусть левая часть системы уравнений имеет вид суммирования по двум индексам
    \begin{equation}
        \sum_{i=1}^n\sum_{k=1}^kx_{ijk} = b_j,\quad j=1,\ldots,m.
    \end{equation}
     Заметим, что
     \begin{equation}
         \sum_{i,j,k}\delta_{sj}x_{ijk}=\sum_{i,k}x_{isk},\quad s=1,\ldots,m.
     \end{equation}
    Тогда, по утверждению \ref{S2}, система данных уравнений после векторизации
    задачи имеет вид
    \begin{equation}
        (I_{1\times k}\otimes E_m \otimes I_{1\times n})\bar{x} = b.
    \end{equation}

    \item Пусть квадратичная форма имеет вид
    \begin{equation}
        \sum_{i=1}^n\sum_{j_1,j_2=1}^mx_{ij_1}x_{ij_2}.
    \end{equation}
     Заметим, что
     \begin{equation}
         \sum_i\sum_{j_1,j_2}x_{ij_1}x_{ij_2}=
        \sum_{i_1,i_2}\sum_{j_1,j_2}\delta_{i_1i_2}x_{i_1j_1}x_{i_2j_2}.
     \end{equation}
      Тогда, по утверждению \ref{S1}, данная квадратичная форма после векторизации задачи имеет вид
      \begin{equation}
          \bar x^T(I_{m\times m}\otimes E_n)\bar x.
      \end{equation}
\end{enumerate}
\end{example}
 \begin{example}     
    Пусть левая часть системы уравнений имеет вид кумулятивной суммы
    \begin{equation}
        \sum_{k=1}^t\sum_{j=1}^mx_{ijk} = b_{it},\quad i=1,\ldots ,n, \enskip t=1,\ldots ,p.
    \end{equation}
    Пусть $D$ - матрица размера $p\times p$ с элементами $d_{ij}=1$ при $i\geq j$ и нулю в противном случае. Тогда
    \begin{equation}
        \sum_{k=1}^t\sum_{j=1}^mx_{ijk}=\sum_{j,k}d_{tk}x_{ijk},\quad t=1,\dots,p.
    \end{equation}
     Заметим, что
     \begin{equation}
\sum_{i,j,k}\delta_{ri}d_{tk}x_{ijk}=\sum_{k=1}^t\sum_{j=1}^mx_{rjk},\quad r=1,\ldots,n, \enskip t=1,\ldots,p.
     \end{equation}
    Тогда, по утверждению \ref{S2}, система данных уравнений после векторизации 
    задачи имеет вид
    \begin{equation}
        (D\otimes I_{1\times m} \otimes E_n)\bar{x} = \bar{b}.
    \end{equation}

\end{example}

\begin{example}
 Пусть квадратичная форма имеет вид циклической суммы
 \begin{equation}
      \sum_{i_1,i_2=1}^n\sum_{j=1}^m a_{i_1i_2}x_{i_1j}x_{i_2,j+1},
 \end{equation}
 где $x_{i,m+1}=x_{i1}$. 
  
  Обозначим через $D$ матрицу размера $m\times m$ с элементами $d_{kl}=0.5$ при $|k-l|=1$ или $|k-l|=m-1$ и $d_{kl}=0$ в противном случае.
 Заметим, что
 \begin{equation}
     \sum_{i_1,i_2}\sum_{j=1}^m a_{i_1i_2}x_{i_1j}x_{i_2,j+1}=
 \sum_{i_1,j_1,i_2,j_2}a_{i_1i_2}d_{j_1j_2}x_{i_1j_1}x_{i_2j_2}. 
 \end{equation}
 Тогда, по утверждению \ref{S1}, данная квадратичная форма после векторизации задачи имеет вид
 \begin{equation}
     \bar{x}^T(D\otimes A)\bar{x}.
 \end{equation}
   
\end{example}

\begin{example}
Специальная замена переменных в квадратичной форме 
\begin{equation}
    \sum_{i,j}q_{ij}y_iy_j.
\end{equation}
Пусть \begin{equation}
    y_i=b_i\sum_{k=1}^m a_kx_{ki}, \quad i=1,\ldots ,n.
\end{equation}

Обозначим через $a$ вектор-строку с элементами $a_k$ и через $D_b$ диагональную матрицу с элементами $b_i$ на диагонали.
По утверждению \ref{S1},
\begin{multline}
 \sum_{i,j}q_{ij}\Bigl ( b_i\sum_{k}a_kx_{ki}\Bigr )\Bigl (b_j\sum_{l}a_lx_{lj}\Bigr ) =\ldots \\
 =\sum_{i,j,k,l}(a_ka_l)(b_iq_{ij}b_j)x_{ki}x_{lj}=\bar{x}^T(D_bQD_b\otimes a^Ta)\bar{x}.   
\end{multline}

Подобные замены возникают при переходе к бинарным переменным.
\end{example}

\section{Преобразование множества значений переменных}\label{p3}
Многие практически значимые задачи являются смешанными, т.е. содержат в своей постановке количественные переменные различных типов. Тип количественной переменной определяется множеством значений, которые она может принимать. Ниже мы рассмотрим способы бинаризации множеств значений ограниченных переменных следующих типов:  целочисленные, дискретные и непрерывные.
Требование ограниченности переменных является естественным для практических задач оптимизации. Для непрерывных переменных дополнительно требуется указать абсолютную погрешность $\epsilon >0$  нахождения их значений.

Под бинаризацией множества значений переменных мы будем понимать аффинное отображение
$L\bar{y}+g$ из множества $Y$ бинарных векторов фиксированной длины  во  множество значений $X$ переменных смешанной задачи.
В случае целочисленных и дискретных переменных отображение бинаризации должно быть сюръективным, т.е. любая точка из  $X$ должна являться образом хотя бы одного бинарного вектора из $Y$. Сюръективность нужна, чтобы не потерять оптимальное решение смешанной задачи после перехода к бинарным переменным.
В случае непрерывных переменных образ всех бинарных векторов $Y$ должен обладать следующим свойством: для любого $\bar{x}^1\in X$ найдется $\bar{x}^2=L\bar{y}+g, \bar{y}\in Y$ такой, что 
$\underset{j}{\max}|\bar{x}_j^1-\bar{x}_j^2|\leq \epsilon$ . Данное требование необходимо для нахождения непрерывной части решения с требуемой точностью.

Пусть аффинное отображении бинаризации $\bar{x}=L\bar{y}+g$ построено. Сделаем замену переменных в целевой функции задачи 
\begin{multline}
f(\bar{x})=\frac{1}{2}\bar{x}^TQ\bar{x}+v^T\bar{x}=\frac{1}{2}(L\bar{y}+g)^TQ(L\bar{y}+g)+v^T(L\bar{y}+g)=\ldots \\
=\frac{1}{2}\bar{y}^T(L^TQL)\bar{y}+\frac{1}{2}\bar{y}^T(L^TQ)g+\frac{1}{2}g^T(QL)\bar{y}+ v^TL\bar{y}+f(g).
\end{multline}
Заметим, что в силу симметричности матрицы $Q$, числа $\bar{y}^T(L^TQ)g$  и $g^T(QL)\bar{y}$  равны. Действительно,
\begin{equation}
    (\bar{y}^T(L^TQ)g)^T=g^T(Q^TL)\bar{y}=g^T(QL)\bar{y}.
\end{equation}
 В результате получаем следующую формулу замены переменных в целевой функции
 \begin{equation}
     f(\bar{x})=\frac{1}{2}\bar{y}^T(L^TQL)\bar{y}+(L^T(Qg+v))^T\bar{y}+f(g).
 \end{equation}
\begin{theorem}\label{S3}   
    Пусть отображение бинаризации имеет вид 
    \begin{equation} 
        x_j=b_j\sum_{i=1}^pa_iy_{ij}+g_j,\quad j=1,\ldots , n,
    \end{equation}
   $D_b$ -- диагональная матрица с элементами $b_j$ на диагонали и  $a$ -- вектор-строка с элементами $a_i$. Тогда     
    \begin{equation}
        \bar{x}=L\bar{y}+g=(D_b\otimes a)\bar{y} +g.
    \end{equation}
   Формула замены переменных в целевой функции имеет вид
    \begin{equation}
        f(\bar{x})=\frac{1}{2}\bar{y}^T(D_bQD_b\otimes a^Ta)\bar{y}+(D_b(Qg+v)\otimes a^T)^T\bar{y}+f(g).
    \end{equation}
 \end{theorem}   
    \begin{proof}
        Равенство $L=D_b\otimes a$ следует из утверждения \ref{S2}. Используя свойства произведения Кронекера, получим
        \begin{multline}
   L^TQL=(D_b\otimes a)^TQ(D_b\otimes a)= 
  (D_b\otimes a^T)(Q\otimes 1)(D_b\otimes a)= \ldots \\
  =(D_bQ\otimes a^T)(D_b\otimes a)=
   D_bQD_b\otimes a^Ta
        \end{multline}
        и
     \begin{equation}
        L^T(Qg+v)=(D_b\otimes a)^T((Qg+v)\otimes 1)=D_b(Qg+v)\otimes a^T.
     \end{equation}
     Что и требовалось доказать.
    \end{proof}
Вектор $b$ часто называют вектором масштабирования, вектор $a$ -- вектором бинаризации и вектор $g$ -- вектором сдвига.

\subsection{Спиновые переменные}
Пусть $\bar{x}=(x_1,\ldots,x_s)^T$ -- вектор-столбец спиновых переменных $x_j\in \{-1,1\}$.
Aффинное отображение бинаризации имеет простой вид $x=2E_sy-I_{s\times 1}$, где $y$ -- бинарный вектор-столбец размерности $s$. Рассмотрим задачу Изинга
\begin{equation}
f(x)=-\frac12x^TJx-h^Tx\to \min,\quad x\in \{-1,1\}^s.
\end{equation}
Вектор бинаризации $a$ равен 1. В силу утверждения \ref{S3}, получаем следующую постановку задачи Изинга в форме QUBO
\begin{multline}
 f(2E_sy-I_{s\times 1})=\frac12y^T(-4J)y+(2JI_{s\times 1}-2h)^Ty+\ldots \\
 +f(-I_{s\times 1})=\frac12y^TQy+v^Ty+const\to \min,
\end{multline}
где
$q_{ij}=-4J_{ij},  v_i=-2h_i+2\sum_{j=1}^sJ_{ij}$ и $const=-\frac12\sum_{i,j}J_{ij}+\sum_{i=1}^sh_i$.

\subsection{Целочисленные переменные}\label{p31}
Пусть $\bar{x}=(x_1,\ldots,x_s)^T$ -- вектор-столбец целочисленных переменных.
Пусть значения целочисленной переменной $x_j\in [n_j,m_j]$, где $n_j$ и $m_j$ -- целые числа. Можно считать, что $m_j-n_j>1$. Действительно, если $m_j-n_j=1$, то $x_j=n_j+y_j, y_j\in \{0,1\}$ и переменная становится бинарной после сдвига ее значений.
Идея построения аффинного отображения основана на возможности двоичного представления целых чисел. Хорошо известно следующее утверждение. Множество всех бинарных векторов размерности $p$  взаимно однозначно отображается на множество всех целых чисел из отрезка $[0,2^p-1]$ по формуле
\begin{equation}
x=\sum_{i=1}^{p}2^{i-1}y_i, \quad y_i\in\{0,1\}.
\end{equation}
Нам понадобится подобное отображение на отрезок $[0,m_j-n_j]$.
Пусть $p_j$ -- минимальное натуральное число, удовлетворяющее неравенству 
\begin{equation}
m_j-n_j\leq 2^{p_j}-1 \Leftrightarrow p_j\geq \log_2\left (m_j-n_j+1\right).
\end{equation}
Тогда   $p_j=\lceil\log_2(m_j-n_j+1)\rceil$. Здесь $\lceil x\rceil$ -- наименьшее целое число, большее или равное $x$. Из условия $m_j-n_j>1$ следует, что $p_j>1$. В силу определения числа $p_j$, верны неравенства
\begin{equation}
2^{p_j-1}-1<m_j-n_j\leq 2^{p_j}-1
\end{equation}
и все целые числа из отрезка $[0,m_j-n_j]$ , и только они, представляются в виде линейной комбинации $p_j$ бинарных переменных $y_{1j}, \ldots ,y_{p_jj}$ следующим образом
\begin{equation}
\sum_{i=1}^{p_j-1}2^{i-1}y_{ij}+(m_j-n_j-2^{p_j-1}+1)y_{p_jj}.
\end{equation}
Действительно, пусть $k\in[0,m_j-n_j]$. Если $k\leq 2^{p_j-1}-1$, то возможно представление данного числа в виде \(k=\sum_{i=1}^{p_j-1}2^{i-1}y_{ij}\). Если $k>2^{p_j-1}-1$, то $k=k_1+(m_j-n_j-2^{p_j-1}+1)$, где $k_1\leq 2^{p_j-1}-1$ и может быть представлено в виде \(k_1=\sum_{i=1}^{p_j-1}2^{i-1}y_{ij}\). 
Заметим, что каждое число из отрезка 
\begin{equation}
[m_j-n_j-2^{p_j-1}+1,2^{p_j-1}-1]
\end{equation} имеет два различных представления в указанном виде.
Таким образом, отображение бинаризации имеет вид
\begin{equation}
x_j=n_j+\sum_{i=1}^{p_j-1}2^{i-1}y_{ij}+(m_j-n_j-2^{p_j-1}+1)y_{p_jj},\quad j=1,\ldots ,s.
\end{equation}
Построенное отображение является сюръективным, но не является взаимно однозначным, разные бинарные вектора могут представлять одно и тоже целое число. 

Выпишем  матричное представление построенного отображения бинаризации.
Пусть $g$ -- вектор-столбец сдвигов значений переменных $(n_1,...,n_s)^T$.
Обозначим через $P_j$ вектор-строку \((1,2,\ldots ,2^{p_j-2},m_j-n_j-2^{p_j-1}+1 )\) коэффициентов представления
$j$-ой целочисленной переменной в виде линейной комбинации бинарных переменных.
Запишем все бинарные переменные в виде вектора-столбца следующим образом
\begin{equation}
\bar{y}=(y_{11},\ldots,y_{p_11},y_{12},\ldots ,y_{p_22},\ldots ,y_{1s},\ldots ,y_{p_ss})^T.
\end{equation}
Тогда $\bar{x}=L\bar{y}+g$, где матрица $L$ размера $s\times m, m=\sum_{j=1}^sp_j$ имеет следующий вид
\begin{equation}
L=
\begin{pmatrix}
P_1     & 0      &\ldots & 0      \\
0       & P_2    &\ldots & 0      \\
\vdots  & \vdots &\ddots & \vdots \\
0       & 0      &\ldots & P_s
\end{pmatrix}
\end{equation}
Если количество значений $m_j-n_j+1$ для всех целочисленных переменных одинаковое, то все $P_j=P$ и матрица $L=E_s\otimes P$.

\subsection{Дискретные переменные}
Пусть $\bar{x}=(x_1,\ldots,x_s)^T$ -- вектор-столбец дискретных переменных.
Пусть значения дискретной переменной $x_j\in \{a_{1j},\ldots ,a_{p_jj}\}$, где $p_j$  -- натуральное число, $p_j>1$ и все числа $a_{ij}$ попарно различны. 
Определим отображение
\begin{equation}
x_j=\sum_{i=1}^{p_j}a_{ij}y_{ij},\quad y_{ij}\in \{0,1\}.
\end{equation}
Тогда, отображение бинаризации это сужение данного отображение на множество решений системы уравнений
\begin{equation}
\sum_{i=1}^{p_j}y_{ij}=1,\quad j=1,\ldots ,s.
\end{equation}
В машинном обучении подобный способ представления категориальных переменных называется one-hot encoding.

Обозначим через $a_j$ вектор-строку \((a_{1j},\ldots ,a_{p_jj} )\) коэффициентов представления
$j$-ой дискретной  переменной в виде линейной комбинации бинарных переменных.
Запишем все бинарные переменные в виде вектора-столбца следующим образом
\begin{equation}
\bar{y}=(y_{11},\ldots,y_{p_11},y_{12},\ldots ,y_{p_22},\ldots ,y_{1s},\ldots ,y_{p_ss})^T.
\end{equation}
Тогда $\bar{x}=L\bar{y}$, при условии $A\bar y=I_{s\times 1}$. Здесь матрицы $L$ и $A$ размера $s\times m, m=\sum_{j=1}^sp_j$ имеют следующий вид
\begin{equation}
L=\begin{pmatrix}
a_1     & 0      &\ldots & 0      \\
0       & a_2    &\ldots & 0      \\
\vdots  & \vdots &\ddots & \vdots \\
0       & 0      &\ldots & a_s
\end{pmatrix}
\quad \textrm{и} \quad 
A=\begin{pmatrix}
I_{1\times p_1}     & 0      &\ldots & 0      \\
0       & I_{1\times p_2}    &\ldots & 0      \\
\vdots  & \vdots &\ddots & \vdots \\
0       & 0      &\ldots & I_{1\times p_s}
\end{pmatrix}.
\end{equation}
Если множества значений всех дискретных переменных совпадают, то все $a_j=a$ и все $p_j=p$, матрица $L=E_s\otimes a$, а дополнительные ограничения запишутся в виде 
$(E_s\otimes I_{1\times p})\bar y=I_{s\times 1}$.

\subsection{Непрерывные переменные}
Пусть $\bar{x}=(x_1,\ldots,x_s)^T$ -- вектор-столбец непрерывных переменных.
Пусть значения непрерывной переменной $x_j\in [a_j,b_j]$, где $a_j$ и $b_j$ -- вещественные числа, $b_j >a_j$ и $d_j=b_j-a_j$.
Зададим абсолютную погрешность $\epsilon >0$  для поиска приближенных значений непрерывных переменных. Отметим, что в практических задачах выбор точности приближенного решения зависит от выбранных единиц измерения реальных величин.
Выберем минимальное натуральное число $p_j$ так, что
\begin{equation}
\frac{d_j}{2^{p_j}-1}\leq 2\epsilon \Leftrightarrow p_j\geq \log_2\left (\frac{d_j}{2\epsilon}+1\right ).
\end{equation}
Тогда  $p_j=\lceil \log_2\left (\frac{d_j}{2\epsilon}+1\right )\rceil$.
Возьмем на отрезке $[a_j,b_j]$ сетку из $2^{p_j}$ точек с равноотстоящими узлами
\begin{equation}
x_{kj}=a_j+\frac{d_j}{2^{p_j}-1}k,\quad k=0, \ldots ,2^{p_j}-1.
\end{equation}
Заметим, что для любого числа $a\in [a_j,b_j]$ найдется точка $x_{kj}$ такая, что $|x_{kj}-a|\leq\epsilon$.
Представляя числа $k\in [0,2^{p_j}-1]$ в двоичном виде $\sum_{i=1}^{p_j}2^{i-1}y_{ij}$, получим отображение бинаризации
\begin{equation}
x_j=a_j+\frac{d_j}{2^{p_j}-1}\sum_{i=1}^{p_j}2^{i-1}y_{ij},\quad y_{ij}\in \{0,1\}.
\end{equation}
Пусть  $P_j=(1,2,\ldots ,2^{p_j-1})$. В матричном виде отображение записывается в виде
\begin{equation}
\bar{x}=L\bar{y}+a,
\end{equation}
где $a=(a_1,\ldots ,a_s)^T$ ,  $\bar{y}=(y_{11},\ldots,y_{p_11},y_{12},\ldots ,y_{p_22},\ldots ,y_{1s},\ldots ,y_{p_ss})^T$ и
\begin{equation}
L=
\begin{pmatrix}
\frac{d_1}{2^{p_1}-1}P_1     & 0      &\ldots & 0      \\
0       & \frac{d_2}{2^{p_2}-1}P_2    &\ldots & 0      \\
\vdots  & \vdots &\ddots & \vdots \\
0       & 0      &\ldots & \frac{d_s}{2^{p_s}-1}P_s
\end{pmatrix}.
\end{equation}
Пусть $D$ -- диагональная матрица размера $s$ с элементами $\frac{d_j}{2^{p_j}-1}$ на диагонали. Если все $p_j$ одинаковые, то все $P_j=P$ и матрица $L=D\otimes P$.

При большой требуемой точности приближенного решения количество переменных многократно увеличивается. Решением проблемы может быть применение для решения задачи итерационного алгоритма, на каждой итерации которого область поиска уменьшается.

\subsection{Примеры преобразований}
В данном параграфе мы рассмотрим задачу
\begin{equation}\label{eq0}
\frac12x^TQx+v^Tx\to \min     
\end{equation}
с различными множествами значений переменных.
\begin{example}
    Пусть в задаче (\ref{eq0}) целочисленные переменные могут принимать только три значения: $-1, 0$ и $1$. 

    Тогда $p=\lceil\log_23\rceil=\lceil 1.585\rceil=2$, вектор бинаризации $P=(1,1)$, размерность бинарного вектора $\bar y$ равна $2n$, где $n$ -- размерность исходной задачи. Вектор сдвига $g=I_{n\times 1}$.
    Отображение бинаризации имеет вид $\bar{x}=(E_n\otimes P)\bar{y}+I_{n\times 1}$. По утверждению \ref{S3}, формула замены переменных в целевой функции записывается в виде
\begin{equation}
f(\bar{x})= \frac{1}{2}\bar{y}^T(Q\otimes P^TP)\bar{y}+((QI_{n\times 1}+v)\otimes P^T)^T\bar{y} + f(g),
\end{equation}
где $f(g)=\frac12\sum_{i,j}q_{ij}+\sum_iv_i$.
\end{example}
\begin{example}
Пусть в задаче (\ref{eq0}) целочисленные переменные  принимают значения от $0$ до $300$. 

Тогда $p=\lceil log_2301\rceil=\lceil 8.2336\rceil=9$, $P=(1,2,\ldots ,128,45)$, размерность бинарного вектора $\bar y$ равна $9n$, где $n$ -- размерность исходной задачи.
Отображение бинаризации имеет вид $\bar{x}=(E_n\otimes P)\bar{y}$. По утверждению \ref{S3}, формула замены переменных в целевой функции
\begin{equation}
f(\bar{x})= \frac{1}{2}\bar{y}^T(Q\otimes P^TP)\bar{y}+(v\otimes P^T)^T\bar{y}.
\end{equation}
\begin{example}
  Пусть в задаче (\ref{eq0}) дискретные переменные могут принимать только три значения: $0, 1$ и $4$. 

    Тогда $p=3$, вектор бинаризации $a=(0,1,4)$, размерность бинарного вектора $\bar y$ равна $3n$, где $n$ -- размерность исходной задачи. 
    Отображение бинаризации имеет вид $\bar{x}=(E_n\otimes a)\bar{y}$.  Дополнительные ограничения имеют вид $(E_n\otimes I_{1\times 3})\bar y=I_{n\times 1}$. По утверждению \ref{S3}, формула замены переменных в целевой функции записывается в виде
\begin{equation} 
f(\bar{x})= \frac{1}{2}\bar{y}^T(Q\otimes a^Ta)\bar{y}+(v\otimes a^T)^T\bar{y}.
\end{equation}
\end{example}
\begin{example}
    Пусть в задаче (\ref{eq0}) непрерывные переменные  принимают значения от $0$ до $100$  и требуемая точность нахождения значений переменных $\epsilon =0.01$. 

Тогда $d=100$, 
\begin{equation}
p=\Bigl \lceil log_2\left (\frac{d}{2\epsilon}+1\right)\Bigr\rceil=
\lceil\log_2(5001)\rceil=\lceil 12.288\rceil=13,
\end{equation}
$P=(1,2,\ldots ,2^{11},2^{12})$, размерность бинарного вектора $\bar y$ равна $13n$, где $n$ -- размерность исходной задачи.
Отображение бинаризации имеет вид $\bar{x}=(\frac{d}{2^p-1}E_n\otimes P)\bar{y}$. По утверждению \ref{S3}, формула замены переменных в целевой функции записывается в виде
\begin{equation}
f(\bar{x})= \frac{1}{2}\bar{y}^T\left(\left(\frac{d}{2^p-1}\right )^2Q\otimes P^TP\right)\bar{y}+(v\otimes P^T)^T\bar{y}.
\end{equation}
\end{example}
\end{example}
\section{Включение линейных ограничений в целевую функцию}\label{p4}
После перехода к одномерному массиву бинарных переменных общая постановка исходной задачи имеет вид
\begin{equation}\label{eq1}
\frac{1}{2}x^TQx+v^Tx\to \min, 
\end{equation}
\begin{equation}\label{eq2}
Ax=b, 
\end{equation}   
\begin{equation}\label{eq3}
Cx\leq d, 
\end{equation}
\begin{equation*}
x_i\in \{0,1\},\quad i=1,\ldots,n, 
\end{equation*}
где $A$ и $C$ -- матрицы размера $m\times n$ и $l\times n$, соответственно.
Мы будем считать, что элементы матриц $A$ и $B$, а также векторов $b$ и $d$ являются целыми числами. Дело в том, что  во многих практических задачах коэффициенты равенств и правые части либо целые, либо заданы с некоторой точностью. Например, с точностью до $k$-го знака после запятой и, после умножения равенств на $10^k$, их коэффициенты и правые части становятся целыми числами.
В данном параграфе рассмотрено сведение задачи (\ref {eq1}) -- (\ref{eq3}) к задаче в форме QUBO.

\subsection{Предварительный анализ ограничений}
Некоторые из уравнений (\ref{eq2}) могут не иметь решения. Действительно, рассмотрим $i$-ое равенство 
\(\sum_ja_{ij}x_j=b_i\). Пусть \(a_i^+\) -- сумма положительных коэффициентов, $a_i^-$ -- сумма отрицательных коэффициентов $a_{ij}$. Если $a_i^+<b_i$, либо $a_i^->b_i$, то равенство никогда не выполняется и задача не имеет решения.

Аналогично, некоторые из неравенств могут либо выполняться всегда, либо не иметь решений. Рассмотрим $i$-ое неравенство $\sum_jc_{ij}x_j\leq d_i$. Пусть $c_i^+$ -- сумма положительных коэффициентов и $c_i^-$ -- сумма отрицательных коэффициентов $c_{ij}$. Если $c_i^+<d_i$, то неравенство выполняется всегда и его можно исключить из постановки задачи. Если $c_i^->d_i$, то неравенство никогда не выполняется и задача не имеет решения.

Заметим, что при решении практических задач имеет смысл предварительно решить задачу
\[\frac{1}{2}\|Ax-b\|^2\to \min.\]
Если оптимальное значение целевой функции будет положительным, то исходная задача не имеет решения, так как не существует решений у системы уравнений (\ref{eq2}).

\subsection{Включение равенств в целевую функцию}
\label{p42}
Рассмотрим сначала простейший вид ограничений (\ref{eq2}). Пусть часть бинарных переменных имеют фиксированные
значения $x(ind)=b$, где $ind$ -- индекс переменных с фиксированными значениями. Введем обозначение для переменных с неизвестными значениями $y=x(\overline{ind})$. 
Легко проверяется, что исходная целевая функция (\ref{eq1}) может быть записана в виде
\begin{equation}\label{eq31}
\frac12y^TQ_yy+(Q_{y,b}b+v_y)^Ty+\frac12b^TQ_bb+v_b^Tb,
\end{equation}
где
\[Q_y=Q(\overline{ind},\overline{ind}),\quad Q_{y,b}=Q(\overline{ind},ind), \quad Q_b=Q(ind,ind),\]
\[v_y=v(\overline{ind}),\quad v_b=v(ind).\]
Представление целевой функции в виде (\ref{eq31}) удобно применять в алгоритмах локального поиска оптимального решения для задач большой размерности. Простейший алгоритм заключается в следующем. На каждой итерации, случайно или по определенному правилу, фиксируются значения большей части переменных и минимизируется целевая функция (\ref{eq31}). В случае уменьшения значения целевой, обновляется текущее решение.

В общем случае, способ включения равенств (\ref{eq2}) в целевую функцию хорошо известен.
В целевую функцию, в качестве слагаемого,  добавляется штрафная функция $\frac{\rho}{2}\|Ax-b\|^2=\rho g(x)$, где $\rho >0$ -- штрафной коэффициент.  Выполняя преобразования, получим
\begin{equation}
g(x)=\frac{1}{2}(Ax-b)^T(Ax-b)=\frac{1}{2}x^T(A^TA)x-(A^Tb)^Tx+\frac{1}{2}\|b\|^2
\end{equation}
и QUBO постановка задачи (\ref{eq1}), (\ref{eq2}), запишется в виде
\begin{equation}
\label{eq4}
\frac{1}{2}x^T(Q+\rho A^TA)x+(v-\rho A^Tb)^Tx +\frac{\rho}{2}\|b\|^2\to \min,
\end{equation}
\[x_i\in \{0,1\}, i=1,\ldots,n.\]
Покажем эквивалентность задач (\ref{eq4}) и (\ref{eq1}), (\ref{eq2}). 

\begin{theorem}\label{S4}
Пусть существует оптимальное решение задачи (\ref{eq1}), (\ref{eq2}). 
Тогда, существует значение $\rho_0>0$ такое, что для всех $\rho\geq \rho_0$ задачи (\ref{eq4})  и (\ref{eq1}), (\ref{eq2}) эквивалентны.
Если элементы $A$ и $b$ целые числа, то справедлива оценка
\begin{equation}
\rho_0\leq \sum_{i,j}|q_{ij}|+2\sum_i|v_i|+2.
\end{equation}
\end{theorem}
\begin{proof}
  Достаточно доказать, что  существует значение $\rho_0>0$ такое, что для всех $\rho\geq \rho_0$   оптимальное решение задачи (\ref{eq4}) является решением системы уравнений (\ref{eq2}).
  
Пусть $E_1=\{y'_1,\ldots,y'_r\}$ и $E_2=\{y''_1,\ldots,y''_s\}$-- упорядоченные по возрастанию множества значений функций $\frac{1}{2}x^TQx+v^Tx$  и  $\frac{1}{2}\|Ax-b\|^2$, соответственно.  Для целевой функции задачи (\ref{eq4}) введем обозначение
\begin{equation}
g(x,\rho)=\frac{1}{2}x^TQx+v^Tx+\frac{\rho}{2}\|Ax-b\|^2.
\end{equation}
  Так как оптимальное решение задачи (\ref{eq1}), (\ref{eq2}) существует, то $y''_1=0$ и $y''_2>0$. 
  Возьмем 
  \begin{equation}
  \rho_0=\frac{y'_r-y'_1+1}{y''_2}\quad \textrm{и} \quad \rho\geq \rho_0.
  \end{equation}
  Пусть $\hat x$ -- оптимальное решение задачи (\ref{eq4}). Докажем, что  $A\hat x=b$. Действительно, предположим $A\hat x\ne b$  и пусть $Ax'=b$. Тогда 
   \begin{equation}
   g(x',\rho)-g(\hat x,\rho)\leq y'_r-y'_1-\rho y''_2\leq y'_r-y'_1-\rho_0 y''_2=-1.
   \end{equation}
Получили противоречие с оптимальностью решения $\hat x$.  Поэтому, $A\hat x=b$.
Оценим сверху число $\rho_0$. Заметим, что
\begin{multline}
y'_r-y'_1=\underset{x',x''}{\max}\Bigl |\frac12\sum_{i,j}q_{ij}(x'_ix'_j-x''_ix''_j)+\sum_iv_i(x'_i-x''_i)\Bigr |\leq \ldots \\
\leq \frac{1}{2}\sum_{i,j}|q_{ij}|+\sum_i|v_i|.
\end{multline}
Так как у матрицы $A$ и вектора $b$ все элементы целые числа, то все элементы вектора $Ax-b$ являются целыми числами. Поэтому, $y''_2\geq \frac12$. Получаем следующую оценку
\begin{equation}
\rho_0=\frac{y'_r-y'_1+1}{y''_2}\leq \sum_{i,j}|q_{ij}|+2\sum_i|v_i|+2.
\end{equation}
Утверждение доказано.
\end{proof}

\subsection{Включение неравенств в целевую функцию}\label{p43}
С помощью вспомогательного неотрицательного целочисленного вектора-столбца $z$ размера $l$, запишем систему неравенств (\ref{eq3}) в виде системы уравнений $Cx+z=d$.  Данная система уравнений определяет множества значений координат вектора $z$. Действительно, $z_i\in [0,d_i-c_i^-]$. 
Построение отображения бинаризации для целочисленных переменных подробно описано в параграфе \ref{p31}. Напомним, что сначала определяются минимальные натуральные числа $p_i$, для которых выполняются неравенства $d_i-c_i^-\leq 2^{p_i}-1$. Затем определяются вектора бинаризации 
\begin{equation}
P_i=(1,2,\ldots,2^{p_i-2},d_i-c_i^--2^{p_i-1}+1)
\end{equation}
и отображение бинаризации $z=L\bar y$, где $\bar y$ -- бинарный вектор-столбец размерности $m=\sum_ip_i$ и 
\begin{equation}
L=
\begin{pmatrix}
P_1     & 0      &\ldots & 0      \\
0       & P_2    &\ldots & 0      \\
\vdots  & \vdots &\ddots & \vdots \\
0       & 0      &\ldots & P_l
\end{pmatrix}.
\end{equation}
В целевую функцию, в качестве слагаемого,  добавляется штрафная функция $\frac{\rho}{2}\|Cx+L\bar y-d\|^2=\rho g(x,\bar y)$, 
где $\rho >0$ -- штрафной коэффициент.
Сформулируем  простое полезное утверждение.
\begin{theorem}\label{S6}
    Пусть вектор-столбец переменных $x=[y;z]$, т.е. все переменные разбиты на две группы переменных $y$ и $z$. Тогда
    \[\frac12\|Ay+Bz-d\|^2=\frac12 x^TSx+R^Tx+\frac12 \|d\|^2,\]
    где
    \[S=\begin{pmatrix}
        A^TA&A^TB\\
        B^TA&B^TB
    \end{pmatrix}
    \quad \textrm{и} \quad
    R=\begin{pmatrix}
        -A^Td\\
        -B^Td
    \end{pmatrix}.\]
\end{theorem}
\begin{proof}
    Выполняя преобразования, получим
\begin{multline}
    \frac12\|Ay+Bz-d\|^2=\frac12(Ay+Bz-d)^T(Ay+Bz-d)=\ldots \\ 
    =\frac12y^T(A^TA)y+y^T(A^TB)z +\frac12z^T(B^TB)z - \ldots\\
    -(A^Td)^Ty-(B^Td)^Tz +\frac12\|d\|^2.
\end{multline}
Из полученного равенства следует, что квадратичная форма и линейная часть имеют требуемый вид.
\end{proof}

Обозначим через $z$ вектор-столбец исходных и вспомогательных переменных $(x;\bar y)$. Тогда,
в силу утверждения \ref{S6}, QUBO постановку задачи (\ref{eq1}), (\ref{eq3}), без учета константы, можно записать в виде 
\begin{equation}
\frac12 z^TSz+R^Tz\to \min,
\end{equation}
\begin{equation}
\label{eq5}
S=
\begin{pmatrix}
Q+\rho C^TC     & \rho C^TL \\
\rho L^TC     & \rho L^TL 
\end{pmatrix},
\quad R=
\begin{pmatrix}
v-\rho C^Td\\
-\rho L^Td  
\end{pmatrix},
\end{equation}

\[z=(x;\bar y),\quad x_i\in \{0,1\},\enskip i=1,\ldots,n,\quad \bar y_k\in \{0,1\},\enskip k=1,\ldots,m.\]

Отметим, что при включении в целевую функцию некоторых простых неравенств можно не вводить вспомогательные переменные и, тем самым, не увеличивать размерность задачи в форме QUBO.  Множество решений каждого такого неравенства совпадает с множеством, на котором обращается в ноль неотрицательная квадратичная функция \cite{3}. Например, множество решений неравенства $\sum_ix_i\leq 1$ совпадает с множеством, на котором обращается в ноль квадратичная функция $\sum_{i\ne j}x_ix_j$. Чтобы учесть данное неравенство в целевой функции, достаточно добавить в целевую функцию слагаемое $\rho \sum_{i\ne j}x_ix_j, \rho >0$. 
Общий подход к выявлению таких "хороших" неравенств заключается в непосредственной проверке совпадения множества решений неравенства и множества нулей неотрицательной квадратичной функции. В качестве иллюстрации данного подхода докажем полезное утверждение.
\begin{theorem}\label{S5}
Пусть $x_i\in \{0,1\},i=1,\ldots, n$ и $y\in \{0,1\}$. Тогда 
\begin{equation}
\sum_{i=1}^nx_i\leq y  \Leftrightarrow 
\left\{
\begin{aligned}
\sum_{i\ne j}x_ix_j&=&0,\\
(1-y)\sum_{i=1}^nx_i & = &0.\\
\end{aligned}
\right.
\end{equation}
\end{theorem}
\begin{proof}
    Пусть $\sum_{i=1}^nx_i\leq y$. Если $y=0$, то $x=0$ и оба равенства верны. Если $y=1$, то второе равенство очевидно, а первое равенство, как мы уже отметили выше, тоже верно.
    
    С другой стороны, пусть оба равенства верны. Если $y=0$, то из второго равенства следует, что $x=0$ и неравенство выполняется. Если $y=1$, то из первого равенства следует, что $\sum_{i=1}^nx_i\leq y$. 
\end{proof}
В силу утверждения \ref{S5}, в качестве квадратичного штрафа для данного неравенства можно взять функцию 
\begin{equation}
\rho_1\sum_{i\ne j}x_ix_j +\rho_2(1-y)\sum_{i=1}^nx_i,\quad \rho_1,\rho_2 >0.
\end{equation}

\subsection{Примеры}
\begin{example}
Фиксированное количество ненулевых бинарных переменных.

Постановка задачи:
\begin{equation}
\frac12 x^TQx+v^Tx\to \min,
\end{equation}
\begin{equation}
\sum_{i=1}^nx_i=k,\quad 0<k<n.
\end{equation}
По общей формуле (\ref{eq4}) получим следующую QUBO постановку задачи
\begin{equation}
\frac{1}{2}x^T(Q+\rho I_{n\times n})x+(v-\rho kI_{n\times 1})^Tx +\frac{\rho k^2}{2}\to \min,
\end{equation}
где $\rho >0$ -- штрафной коэффициент. 
\end{example}

\begin{example}
Матричная бинарная переменная. Ограничения на количество ненулевых элементов в столбцах.

Постановка задачи:
\begin{equation}
\sum_{i=1}^n\sum_{j=1}^nc_{ij}x_{ij}\to \min,
\end{equation}
\begin{equation}
\sum_{i=1}^nx_{ij}=k_j,\quad j=1,\ldots,n.
\end{equation}
 Пусть $k$ -- вектор-столбец с элементами $k_j$.
 В силу утверждения \ref{S2}, система уравнений запишется в виде
 \begin{equation} 
 (E_n\otimes I_{1\times n})\bar x=k
 \end{equation}
 и
 \begin{multline}
 \|E_n\otimes I_{1\times n}\bar x-k\|^2=\bar x^T(E_n\otimes I_{n \times n})\bar x - 2((E_n\otimes I_{n\times 1})k)^T\bar x+\|k\|^2=\ldots \\
 =\bar x^T(E_n\otimes I_{n \times n})\bar x - 2(k\otimes I_{n\times 1})^T\bar x+\|k\|^2.
 \end{multline}
 Следовательно, QUBO постановка задачи, без учета константы, имеет вид
 \begin{equation} 
 \frac{\rho}{2} \bar x^T(E_n\otimes I_{n \times n})\bar x +\left (\bar c -\rho (k\otimes I_{n\times 1})\right )^T\bar x \to \min,
 \end{equation}
 где $\rho >0$ -- штрафной коэффициент.
\end{example}

\begin{example}
Включение бинарных неравенств в целевую функцию.

Постановка задачи:
\begin{equation}
\frac12 x^TQx+v^Tx\to \min,
\end{equation}
\begin{equation}
m_j\leq \sum_{i=1}^nc_{ij}x_i\leq M_j,\quad m_j<M_j,\quad j=1,\ldots,s,
\end{equation}
где $m_j$ и $M_j$ -- целые числа, $c_{ij}\in \{0,1\}$.

Запишем двойные неравенства в виде равенств
$\sum_{i=1}^nc_{ij}x_i-y_j=0$, где $y_j$ -- целочисленная переменная со значениями из отрезка $[m_j,M_j]$.
Количество бинарных переменных для представления вспомогательной переменной $y_j$  равно $p_j=\left \lceil \log_2(M_j-m_j+1)\right \rceil$.
Вектора-строки бинаризации имеют вид $P_j=(1,2,\ldots,2^{p_j-2},M_j-m_j-2^{p_j-1}+1)$. 
Отображение бинаризации запишется в виде 
$y=L\bar z +m$, 
где $z$ -- вектор-столбец бинарных переменных длины $p=\sum_{j=1}^sp_j$, 
$m$ -- вектор-столбец нижних границ $(m_1,\ldots,m_s)^T$ и матрица $L$ имеет вид
\begin{equation}
L=
\begin{pmatrix}
P_1     & 0      &\ldots & 0      \\
0       & P_2    &\ldots & 0      \\
\vdots  & \vdots &\ddots & \vdots \\
0       & 0      &\ldots & P_s
\end{pmatrix}.
\end{equation}
Система исходных бинарных неравенств может быть включена в целевую функцию в виде слагаемого вида
\begin{equation}
\frac{\rho}{2}\|Cx-Lz-m\|^2,
\end{equation}
где $\rho >0$ -- штрафной коэффициент.

Пусть вектор-столбец переменных $w=[x;z]$. Тогда, по утверждению \ref{S6}, QUBO форма исходной задачи имеет вид
    \begin{equation}
    \frac12 w^TSw+R^Tw+\frac{\rho}{2} \|m\|^2\to \min,
    \end{equation}
    где
    \[S=\begin{pmatrix}
        Q+\rho C^TC&-\rho C^TL\\
        -\rho L^TC&\rho L^TL
    \end{pmatrix}
    \quad \textrm{и} \quad
    R=\begin{pmatrix}
        v-\rho C^Tm\\
        \rho L^Tm 
    \end{pmatrix}.\]

\end{example}

\section{Примеры приведения задач к форме QUBO}
В данном разделе рассмотрено приведение к форме QUBO нескольких известных задач бинарной оптимизации.

\subsection{Задача о максимальном разрезе (The Max-Cut Problem)}
Это одна из самых знаменитых задач комбинаторной оптимизации.
Пусть \(G=(V,E)\) -- неориентированный граф. Каждому ребру $(i,j)\in E$ приписан вес $c_{ij}$. 
Обозначим через $C$ симметричную матрицу весов. Положим $c_{ii}=0$. Требуется разбить множество вершин на два непересекающихся подмножества $V_1$ и $V_2$ так, чтобы сумма весов 
ребер с концами в разных подмножествах вершин графа была максимальной.
Определим бинарные переменные $x_i,i=1,\ldots,n$. Пусть $x_i=1$, если $i\in V_1$, и $x_i=0$ в противном случае. Следовательно, $x_i+x_j-2x_ix_j=0$ тогда и только тогда, когда $i$ и $j$ содержатся в одном и том же подмножестве вершин графа. Если вершины в разных множествах, то $x_i+x_j-2x_ix_j=1$.   Получаем следующую постановку задачи
\begin{equation}\label{eq7}
\sum_{(i,j)\in E}c_{ij}(x_i+x_j-2x_ix_j)=\frac12 \sum_{i,j=1}^nc_{ij}(x_i+x_j-2x_ix_j) \to \max.
\end{equation}
Переходя к задаче на минимум, получаем постановку задачи о максимальном разрезе в форме QUBO
\begin{equation}
\frac12 x^TQx+v^Tx\to \min,
\end{equation}
где $Q=2C$ и $v$ -- вектор-столбец с координатами $v_i=-\sum_{j=1}^nc_{ij}$.
Делая замену $x_i=\frac{s_i+1}{2}$ в (\ref{eq7}), получим постановку задачи о максимальном разрезе в форме задачи Изинга
\begin{equation}
-\frac12 s^TJs\to \min,
\end{equation}
где $J=-\frac12 C$ и $ s_i\in \{-1,1\}$.

Следует отметить, что к задаче о максимальном разрезе взвешенного графа может быть сведена любая задача в форме QUBO
\begin{equation}
\frac12x^TQx+v^Tx\to \min.
\end{equation}
Сначала данная задача преобразуется к задаче в форме Изинга
\begin{equation}
-\frac12s^TJs-h^Ts + const\to \min,
\end{equation}
\begin{equation}
J_{ij}=-\frac14(1-\delta_{ij})q_{ij}, \quad h_i=-\frac12v_i-\frac14\sum_{j=1}^nq_{ij}
\end{equation}
и
\begin{equation}
const=\frac18\sum_{i\ne j}q_{ij}+\frac14\sum_{i=1}^nq_{ii}+\frac12\sum_{i=1}^nv_i.
\end{equation}
Затем, с помощью дополнительной переменной $s_0$ с фиксированным значением $-1$, задача преобразуется к задаче в форме Изинга с нулевой линейной частью
\begin{equation}
-\frac12s'^TJ's'\to \min, \quad s'=[-1;s].
\end{equation}
Действительно, пусть матрица квадратичной формы имеет вид 
\begin{equation}
J'=\begin{pmatrix}
    0&-h^T\\
    -h&J
\end{pmatrix}.
\end{equation}
Легко проверить, что целевые функции задач совпадают
\begin{equation}
-\frac12s'^TJ's'=-\frac12s^TJs-h^Ts.
\end{equation}
Таким образом, задача в форме QUBO свелась к задаче о максимальном разрезе графа с весовой матрицей ребер $C=-2J'$.  

\subsection{Задача о максимальном независимом множестве (MIS)}

Задача формулируется следующим образом. Дан неориентированный граф  \(G=(V,E)\). Требуется найти максимальное по количеству элементов подмножество $S\subseteq V$ попарно несмежных между собой вершин. Пусть $n$ -- количество вершин графа.
Определим бинарные переменные $x_i,i\in V$. Пусть $x_i=1$, если $i\in S$, и $x_i=0$ в противном случае. Следовательно, $\sum_{(i,j)\in E}x_ix_j=0$ тогда и только тогда, когда $S$ независимое подмножество. Получаем следующую постановку задачи
\begin{equation}
-\sum_{i=1}^nx_i+\alpha \sum_{(i,j)\in E}x_ix_j \to \min,
\end{equation}
где $\alpha >0$ -- штрафной коэффициент.

Пусть $A$ -- матрица смежности графа $G$. Получим следующую постановку задачи в форме QUBO
\begin{equation}
\frac12x^TQx+v^Tx\to \min,\quad Q=\alpha A,\quad v=-I_{n\times 1}.
\end{equation}

Отметим, что множество вершин $S$ независимо тогда и только тогда, когда оно является кликой в дополнении графа $G$.  Следовательно, задача о нахождении максимальной клики
(Maximum clique) в форме QUBO имеет вид
\begin{equation}
\frac12x^TQx+v^Tx\to \min,\quad Q=\alpha \bar A,\quad v=-I_{n\times 1},
\end{equation}
где $\bar A$ -- поэлементное логическое отрицание матрицы смежности графа $G$.

\subsection{Мощность объединения подмножеств (The Min-k-Union problem)}
Задача формулируется следующим образом \cite{24}. Пусть задано конечное множество элементов $V$ и задана система $S$ подмножеств 
данного множества $V$. 
Требуется найти подсистему $S^k\subset S$, состоящую из $k$ подмножеств, объединение которых содержит минимальное количество элементов множества $V$.

Формализуем задачу. Пусть элементы множества $V$ натуральные числа  $v=1:n$. 
Занумеруем подмножества из $S$ натуральными числами $p=1:m$. Система подмножеств $S$ однозначно задается
матрицей $A=(a_{pv})$, где $a_{pv}=1$, если $p$-ое подмножество содержит элемент $v$, и $a_{pv}=0$ в противном случае.  

Введем бинарные переменные $x_p, p=1:m$. Переменная $x_p$ равна, единице если подмножество c номером $p$ содержится в искомой подсистеме $S^k$,
и нулю в противном случае.

Введем бинарные переменные $y_v, v=1:n$. Переменная $y_v$ равна единице, если элемент $v$ содержится в объединении множеств искомой подсистемы $S^k$, и
нулю в противном случае.

Так как искомая подсистема $S^k$ содержит ровно $k$ подмножеств, то должно выполняться равенство
\begin{equation}
\left (k-\sum_{p=1}^mx_p\right )^2=0.
\end{equation}
Из определения бинарных переменных следует, что выполняются неравенства $y_v\geq x_p$, если $p$-ое подмножество содержит элемент $v$.
Данная система неравенств равносильна равенству
\begin{equation}
\sum_{v=1}^n\sum_{p=1}^ma_{pv}(1-y_v)x_p=0.
\end{equation}
Получаем следующую QUBO постановку исходной задачи
\begin{equation}
c_1\left (k-\sum_{p=1}^mx_p\right )^2+c_2\sum_{v=1}^n\sum_{p=1}^ma_{pv}(1-y_v)x_p +c_3\sum_{v=1}^ny_v \to \min,
\end{equation}
где $c_1,c_2$ и $c_3$ положительные штрафные коэффициенты. В работе \cite{24} указаны условия на штрафные коэффициенты, при которых QUBO
постановка задачи эквивалентна исходной задаче. В частности, можно положить $c_3=1$ и $c_1=c_2=n+1$.

Объединим переменные в один вектор $w=[x;y]$. Тогда в матричном виде QUBO постановка задачи имеет вид
\begin{equation}
\frac{1}{2}w^TQw+v^Tw+c_1k^2\to \min,
\end{equation}
где
\[Q=
\begin{pmatrix}
2c_1I_{m\times m} &-c_2A\\
-c_2A^T &0_{n\times n}
\end{pmatrix}
\quad \textrm{и} \quad
v=
\begin{pmatrix}
-2c_1kI_{m\times 1}+c_2g\\
c_3I_{n\times 1} 
\end{pmatrix},
\]
где вектор $g$ получен суммированием матрицы $A$ по строкам, т.е. $g_p=\sum_{v=1}^na_{pv}$.

\subsection{Задача квадратичного назначения (Quadratic Assignment)}

Это известная задача комбинаторной оптимизации
с приложениями в самых разных областях.
Даны $n$ объектов и $n$ локаций, матрица потоков $F=\bigl (f_{i_1i_2}\bigr )$ ресурсов между объектами и
матрица расстояний между локациями  $D=\bigl (d_{j_1j_2}\bigr )$.
Требуется найти размещение объектов по локациям, минимизирующее взвешенный поток через систему.
Определим массив бинарных переменных \(x_{ij},i,j=1,\ldots,n\). Пусть \(x_{ij}=1\), если $i$-ый объект размещен в $j$-ой локации, и \(x_{ij}=0\) в противном случае. 

Постановка задачи:
\begin{equation}
\sum_{i_1,j_1,i_2,j_2}f_{i_1i_2}d_{j_1j_2}x_{i_1j_1}x_{i_2j_2}\to \min
\end{equation}
при условии, что
\begin{equation}
\sum_{i}x_{ij}=1, \quad j=1,\ldots ,n
\end{equation}
и
\begin{equation}
\sum_{j}x_{ij}=1, \quad i=1,\ldots ,n.
\end{equation}
 Из утверждений \ref{S1} и \ref{S2} следует, что после векторизации задача принимает вид
 \begin{equation}
 \bar{x}^T(D\otimes F)\bar{x}\to \min
 \end{equation}
 при условии, что
 \begin{equation}
 (E_n\otimes I_{1\times n})\bar{x}=I_{n\times 1}
 \end{equation}
 и
  \begin{equation}
  (I_{1\times n}\otimes E_n)\bar{x}=I_{n\times 1}.
  \end{equation}
  Квадраты норм разности правых и левых частей полученных систем уравнений могут быть вычислены следующим образом
  \begin{multline}\|(E_n\otimes I_{1\times n})\bar{x}-I_{n\times 1}\|^2=\bar{x}^T(E_n\otimes I_{n\times n})\bar{x} - \ldots \\
  -2((E_n\otimes I_{n\times 1})I_{n\times 1})^T\bar{x}+n=
 \bar{x}^T(E_n\otimes I_{n\times n})\bar{x} - 2I_{n^2\times 1}^T\bar{x}+n
  \end{multline}
  и
  \begin{multline}
  \|(I_{1\times n}\otimes E_n )\bar{x}-I_{n\times 1}\|^2=\bar{x}^T(I_{n\times n}\otimes E_n )\bar{x} - \ldots \\
  -2((I_{n\times 1}\otimes E_n)I_{n\times 1})^T\bar{x}+n=
  \bar{x}^T(I_{n\times n}\otimes E_n )\bar{x} - 2I_{n^2\times 1}^T\bar{x}+n.   
  \end{multline}

\subsection{Задача о раскраске вершин графа (Graph Coloring)}

Пусть есть $m$ цветов и требуется раскрасить $n$ вершин
графа \((V,E)\) так, чтобы любые две смежные вершины были раскрашены в разные цвета и количество использованных цветов было минимальным.
Введем бинарные переменные: \(x_{ij}=1\), если $i$-ая вершина раскрашена в $j$-ый цвет, и \(x_{ij}=0\) в противном случае; \(w_j=1\), если $j$-ый цвет использован для раскраски вершин, и \(w_j=0\) в противном случае. 

Тогда постановка задачи имеет вид \cite{9}
\begin{equation}
\sum_{j=1}^mw_j\to \min
\end{equation}
при условии, что
\begin{enumerate}
    \item Каждая вершина окрашена ровно в один цвет 
    \begin{equation}
    \sum_{i=1}^n \Bigl(\sum_{j=1}^mx_{ij} -1\Bigr)^2=0
    \end{equation}
    \item Вершины, связываемые ребром, должны быть разных цветов 
    \begin{equation}
    x_{i_1j}+x_{i_2j}\leq w_j,\quad (i_1,i_2)\in E, \enskip j=1,...,m.
    \end{equation}
\end{enumerate}
Рассмотрим первое условие. Из утверждений \ref{S1} и \ref{S2} следует, что
\begin{multline}
 \sum_{i=1}^n \Bigl(\sum_{j=1}^mx_{ij} -1\Bigr)^2=\sum_{i,j_1,j_2}x_{ij_1}x_{ij_2}-2\sum_{i,j}x_{ij}+n= \ldots \\
=\sum_{i_1,j_1,i_2,j_2}\delta_{i_1i_2}x_{i_1j_1}x_{i_2j_2}-2\sum_{i,j}x_{ij}+n= \ldots \\
 =\bar{x}^T(I_{m\times m}\otimes E_{n})\bar{x} -2I_{nm\times 1}^T\bar{x}+n=0.
\end{multline}
Рассмотрим второе условие. Обозначим через \(A\) бинарную матрицу смежности вершин графа, определяемую множеством ребер. Из утверждения \ref{S5} следует, что данное условие равносильно одновременному выполнению двух равенств
\begin{equation}
\sum_{i_1,i_2,j}a_{i_1i_2}x_{i_1j}x_{i_2j}=0 \quad \textrm{и}\quad \sum_{i_1,i_2,j}a_{i_1i_2}(1-w_j)(x_{i_1j}+x_{i_2j})=0.
\end{equation}
В силу утверждения \ref{S1}, первое равенство можно записать в виде
\begin{equation}
\sum_{i_1,i_2,j}a_{i_1i_2}x_{i_1j}x_{i_2j}=
\sum_{i_1,j_1,i_2,j_2}a_{i_1i_2}\delta_{j_1j_2}x_{i_1j_1}x_{i_2j_2}=
\bar{x}^T(E_m\otimes A)\bar{x}=0.
\end{equation}
Заметим, что
\begin{multline}
 \sum_{i_1,i_2,j}a_{i_1i_2}(1-w_j)(x_{i_1j}+x_{i_2j})=\ldots \\
=\sum_{j}(1-w_j)\sum_{i_1,i_2}a_{i_1i_2}(x_{i_1j}+x_{i_2j})=2\sum_{j}(1-w_j)\sum_id_ix_{ij},   
\end{multline}
где $d_i$ -- степень $i$-ой вершины графа.
Так как каждая вершина окрашена ровно в один цвет, то
\begin{equation} 
\sum_{j}(1-w_j)\sum_id_ix_{ij}=-2\sum_{i,j}d_iw_jx_{ij}+2\sum_{i}d_i.
\end{equation}
В силу утверждения \ref{S2} и примера \ref{ex22},
\begin{equation}
2\sum_{i,j}d_iw_jx_{ij}=2\sum_jw_j\sum_id_ix_{ij}=2w^T((E_m\otimes d)\bar x),
\end{equation}
где $d$ -- вектор-строка степеней вершин графа.
Пусть $z=[w;\bar x]$ -- вектор-столбец всех переменных задачи. Добавляя ограничения в виде квадратичных штрафов в целевую функцию и учитывая, что $w_j=w_j^2,\bar x_{k}=\bar x_{k}^2$, получим следующую постановку исходной задачи в форме QUBO
\begin{equation}
\frac12z^TQz\to \min,
\end{equation}
где
$$Q=\begin{pmatrix}
    \alpha_0E_m&-\alpha_2(E_m\otimes d)\\
    -\alpha_2(E_m\otimes d^T)&\alpha_1((I_{m\times m}-2E_m)\otimes E_{n})+\alpha_1(E_m\otimes A)
\end{pmatrix}$$
и $\alpha_0,\alpha_1,\alpha_2>0 $ -- штрафные коэффициенты.

\subsection{Задача о рюкзаках (Multiple Knapsack Problem)}

Пусть имеется $n$ контейнеров с номерами $i=1,\ldots,n$, с заданной вместимостью $c_i$, и $m$ предметов с номерами $j=1,\ldots,m$. Каждый предмет $j$ имеет стоимость $f_{ij}$ и вес $w_{ij}$, зависящие от того, в какой контейнер помещен данный предмет. 
Обычно предполагается, что вместимости контейнеров, стоимость и вес предметов имеют целые положительные значения. Требуется заполнить контейнеры предметами с максимальной суммарной стоимостью. 

Введем бинарные переменные $x_{ij}$. Переменная $x_{ij}=1$, если $j$-ый предмет помещен в $i$-ый контейнер, и $x_{ij}=0$ в противном случае.
Получаем следующую постановку задачи:
\begin{equation}
-\sum_{i,j}f_{ij}x_{ij}\to \min
\end{equation}
при условии, что
\begin{equation}
\sum_{i=1}^nx_{ij}\leq 1,\quad j=1,\ldots,m,
\end{equation}
\begin{equation}
\sum_{j=1}^mw_{ij}x_{ij}\leq c_i, \quad i=1,\ldots,n.
\end{equation}
Заметим, что квадратичный штраф для первой системы неравенств, как мы отмечали в параграфе \ref{p43}, может быть выписан без введения вспомогательных переменных и, по утверждению \ref{S1}, имеет вид
\begin{equation}
\frac{\alpha}{2}\sum_{j=1}^m\sum_{i_1\ne i_2}x_{i_1j}x_{i_2j}=\frac{\alpha}{2}\sum_{j_1,j_2}\sum_{i_1\ne i_2}\delta_{j_1j_2}x_{i_1j_1}x_{i_2j_2}=\frac{\alpha}{2}\bar x^T(E_m\otimes (I_{n\times n}-E_n))\bar x,
\end{equation}
где $\alpha >0$ -- штрафной коэффициент.

Обозначим через $D_{\bar w}$ диагональную матрицу с элементами вектора $\bar w$ на диагонали. Пусть $c$ -- вектор-столбец вместимостей контейнеров.
По утверждению \ref{S2}, после упорядочивания переменных, задача принимает вид
\begin{equation}
\frac{\alpha}{2}\bar x^T(E_m\otimes (I_{n\times n}-E_n))\bar x-\bar f^T\bar x\to \min
\end{equation}
при условии, что
\begin{equation}
(I_{1\times m}\otimes E_n)D_{\bar w}\bar x\leq c.
\end{equation}
Преобразуем систему неравенств в систему уравнений с помощью неотрицательного целочисленного вектора-столбца $z$ размера $n$
\begin{equation}
(I_{1\times m}\otimes E_n)D_{\bar w}\bar x+z=c.
\end{equation}
Бинаризируем вектор $z$. Заметим, что $z_i\in [0,c_i]$. Положим $p_i=\lceil\log_2(c_i+1)\rceil$.  Отображение бинаризации имеет вид $z=Ly$, где $y$ -- бинарный вектор-столбец размера $p=\sum_{i}p_i$,
\begin{equation}
L=
\begin{pmatrix}
P_1     & 0      &\ldots & 0      \\
0       & P_2    &\ldots & 0      \\
\vdots  & \vdots &\ddots & \vdots \\
0       & 0      &\ldots & P_n
\end{pmatrix}
\end{equation}
и $P_i=(1,2,\ldots,2^{p_i-2},c_i-2^{p_i-1}+1)$.

Пусть $u=(\bar x;y)$ -- вектор-столбец основных и вспомогательных переменных. Квадратичный штраф для системы неравенств, в силу утверждения \ref{S6}, имеет вид
\begin{equation}
\frac{\beta}{2}\|(I_{1\times m}\otimes E_n)D_{\bar w}\bar x+Ly-c\|^2=\frac{\beta}{2}u^TQ_cu+\beta v_c^Tu+\frac{\beta}{2}\|c\|^2,
\end{equation}
где
$$ Q_c=
\begin{pmatrix}
D_{\bar w}(I_{m\times m}\otimes E_n)D_{\bar w}     & D_{\bar w}(I_{m\times 1}\otimes L)\\
(I_{1\times m}\otimes L^T)D_{\bar w}       & L^TL
\end{pmatrix},
$$
$$
v_c=
\begin{pmatrix}
-D_{\bar w}(I_{m\times 1}\otimes c)\\
-L^Tc
\end{pmatrix}
\quad \textrm{и} \quad \beta >0.
$$
Добавляя в целевую функцию квадратичный штраф в виде слагаемого без учета константы $\frac{\beta}{2}\|c\|^2$, получим постановку задачи в форме QUBO
\begin{equation}
\frac12 u^TQu+v^Tu\to \min,
\end{equation}
где
$$Q=
\begin{pmatrix}
\alpha (E_m\otimes (I_{n\times n}-E_n))+\beta D_{\bar w}(I_{m\times m}\otimes E_n)D_{\bar w} 
&\beta D_{\bar w}(I_{m\times 1}\otimes L)\\
\beta (I_{1\times m}\otimes L^T)D_{\bar w}       & \beta L^TL
\end{pmatrix}
$$
и
$$v=
\begin{pmatrix}
-\bar f - \beta D_{\bar w}(I_{m\times 1}\otimes c)\\
\beta L^Tc
\end{pmatrix}.
$$

\subsection{Транспортная задача (Vehicle Routing Problem)}

Пусть задан набор из $n+1$ транспортных узлов $i=1,\ldots,n+1$ и множество связей между ними $E$.
Первые $n$ узлов это кусты, на каждом из которых нужно провести бурение. 
Бурение проводится с использованием набора $m$ буровых установок $v=1,\ldots,m$.
Каждая буровая установка проходит последовательность узлов длиной $P \leq n + 2$.
Депо -- это $(n+1)$-ый узел, в котором начинаются и заканчиваются все перемещения установок. Для того, чтобы буровая установка могла оставаться в депо после прибытия, множество связей между узлами должно содержать пару $(n+1,n+1)$.
Каждой паре смежных узлов $(i,j)\in E $ сопоставляется стоимость переезда $c_{ij}>0$.
Целью ставится минимизация суммарной стоимости всех переездов буровых установок.

Данная задача допускает различные математические формулировки. 
Ниже приводится формулировка задачи на основе последовательностей (Sequence-based formulation \cite{23}).

В Sequence-based формулировке вводятся бинарные переменные $x_{vpi}$, которые принимают значение 1, если буровая установка $v$ посещает узел
$i$ с позицией $p$ в последовательности узлов в своем маршруте, в противном случае $x_{vpi} = 0$. Целевая функция, требующая минимизации, имеет вид
\begin{equation}
\sum_{v=1}^m\sum_{p=1}^{P-1}\sum_{(i,j)\in E}c_{ij}x_{vpi}x_{v,p+1,j} \to \min
\end{equation}
при условии, что:
\begin{enumerate}
\item
Каждый узел, кроме депо, будет пройден ровно один раз
\begin{equation}
\sum_{v=1}^m\sum_{p=1}^Px_{vpi}=1,\quad i=1,\ldots,n,
\end{equation}
\item
Номер позиции в последовательности узлов однозначно определяет узел
\begin{equation}
\sum_{i=1}^{n+1}x_{vpi}=1,\quad v=1,\ldots,m,\enskip p=1,\ldots,P,
\end{equation}
\item
Последовательные переходы буровых установок возможны только между смежными узлами
\begin{equation}
\sum_{v=1}^m\sum_{p=1}^{P-1}\sum_{(i,j)\notin E}x_{vpi}x_{v,p+1,j}=0,
\end{equation}
\item
Каждая установка бурения после прибытия в депо там и остается
\begin{equation}
\sum_{v=1}^m\sum_{p=2}^{P-1}\sum_{j=1}^nx_{v,p,n+1}x_{v,p+1,j}=0,
\end{equation}
\item
Последовательности узлов начинаются и заканчиваются в депо
\begin{equation}
\sum_{v=1}^m(1-x_{v,1,n+1}x_{v,P,n+1})=0,
\end{equation}
\item Равенство нулю переменных (следствие из условий 2 и 5)
\begin{equation}
\sum_{v=1}^m\sum_{i=1}^n(x_{v1i}+x_{vPi})=0,
\end{equation}
\begin{equation}
\sum_{v=1}^m\sum_{(n+1,j)\notin E}(x_{v2j}+x_{v,P-1,j})=0.
\end{equation}

\end{enumerate}
Если требуется чтобы были задействованы все буровые установки, то нужно добавить условие
равенства нулю суммы
\begin{equation}
\sum_{v=1}^mx_{v,2,n+1}=0.
\end{equation}

Проведем векторизацию задачи. 
Обозначим через $D$ матрицу размера $P\times P$ с элементами $d_{kl}=0.5$ при $|k-l|=1$ и нулю в противном случае.
В силу утверждений \ref{S1} и \ref{S2}, получим следующую постановку задачи
\begin{equation}
\frac12 \bar x^T(C\otimes D\otimes E_m)\bar x \to \min
\end{equation}
при условии, что
\begin{enumerate}
  \item Пусть матрица $M$ получена из $E_{n+1}$ вычеркиванием последней строки. Тогда система уравнений запишется в виде
    \begin{equation}
    (M\otimes I_{1\times P}\otimes I_{1\times m})\bar x=I_{n\times 1},
    \end{equation}
  а соответствующий квадратичный штраф равен
  \small
  \begin{multline} \frac{\alpha_1}{2}\|(M\otimes I_{1\times P}\otimes I_{1\times m})\bar x - I_{n\times 1}\|^2=\ldots \\
  =\frac{\alpha_1}{2}\bar x^T  (M^TM\otimes I_{P\times P} \otimes I_{m\times m})\bar x - \alpha_1\Bigl ( (M^T\otimes I_{P\times 1}\otimes I_{m\times 1})I_{n\times 1}\Bigr )^T\bar x +\ldots \\+\frac{\alpha_1}{2}n =\frac{\alpha_1}{2}\bar x^T(E'_{n+1}\otimes I_{mP\times mP} )\bar x - \alpha_1 ( I'_{(n+1)\times 1}\otimes I_{mP\times 1} )^T\bar x +\frac{\alpha_1}{2}n,
  \end{multline}
  \normalsize
  где матрица $E'_{n+1}$ получена из матрицы $E_{n+1}$ обнулением элемента $(n+1,n+1)$, а вектор $I'_{(n+1)\times 1}$ получен из вектора $I_{(n+1)\times 1}$ обнулением элемента $(n+1,1)$.
  \item Вторая система уравнений запишется в виде 
  \begin{equation}
  (I_{1\times (n+1)}\otimes E_P \otimes E_m)\bar x=I_{mP\times 1},
  \end{equation}
  а соответствующий квадратичный штраф равен 
  \small
  \begin{multline} \frac{\alpha_2}{2}\|(I_{1\times (n+1)}\otimes E_P \otimes E_m)\bar x-I_{mP\times 1}\|^2=\ldots \\
  =\frac{\alpha_2}{2}\bar x^T  (I_{(n+1)\times (n+1)}\otimes E_{P} \otimes E_{m})\bar x - \alpha_2
  \Bigl ( (I_{(n+1)\times 1}\otimes E_{P}\otimes E_{m})I_{mP\times 1}\Bigr )^T\bar x +\ldots \\+\frac{\alpha_2}{2}mP =\frac{\alpha_2}{2}\bar x^T(I_{(n+1)\times (n+1)}\otimes E_{mP} )\bar x - \alpha_2  I_{(n+1)mP\times 1}^T\bar x +\frac{\alpha_2}{2}mP.
  \end{multline}
  \normalsize
  \item Данный квадратичный штраф запишется в виде
  \begin{equation}
  \frac{\alpha_3}{2}\sum_{v=1}^m\sum_{p=1}^{P-1}\sum_{(i,j)\notin E}x_{vpi}x_{v,p+1,j}=\frac{\alpha_3}{2}\bar x^T(\bar A\otimes D\otimes E_m)\bar x,
  \end{equation}
  где $\bar A$ -- логическое отрицание матрицы смежности узлов.
  \item Определим множество пар индексов 
  \small
  $$W=\{(k,l):k=nP+p,\quad l=(j-1)P+p+1,\quad p=2,\ldots,P-1, \quad j=1,\ldots,n\}.$$
  \normalsize
  Пусть $G$ -- матрица размера $(n+1)P\times (n+1)P$. Положим  $g_{kl}$ и $g_{lk}$ равными $0.5$ при $(k,l)\in W$ и нулю в противном случае. Тогда квадратичный штраф запишется в виде
   \begin{equation}
   \frac{\alpha_4}{2}\sum_{v=1}^m\sum_{p=2}^{P-1}\sum_{j=1}^nx_{v,p,n+1}x_{v,p+1,j}=\frac{\alpha_4}{2}\bar x^T(G\otimes E_m)\bar x.
   \end{equation}
  \item Пусть элементы $(1,P)$ и $(P,1)$ матрицы $H$ размера $P\times P$ равны $0.5$, а остальные элементы данной матрицы равны нулю. Пусть $E''_{n+1}$ -- матрица размера $(n+1)\times (n+1)$ у которой элемент $(n+1,n+1)$ равен 1, а остальные элементы равны нулю. Тогда, квадратичный штраф запишется в виде
   \begin{equation}
   \frac{\alpha_5}{2}\sum_{v=1}^m(1-x_{v,1,n+1}x_{v,P,n+1})=-\frac{\alpha_5}{2}\bar x^T(E''_{n+1}\otimes H\otimes E_m)\bar x+\frac{\alpha_5}{2}m.
   \end{equation}
   \item Равенство нулю переменных можно учесть следующим образом. В QUBO матрице вычеркнуть строки и столбцы с соответствующими номерами. После решения задачи, в найденное решение вставить недостающие нули. Сформируем индекс нулевых элементов вектора $\bar x$.

   Рассмотрим первую группу нулевых переменных
   \begin{equation}
   \sum_{v=1}^m\sum_{i=1}^n(x_{v1i}+x_{vPi})=0.
   \end{equation}
   Пусть $e^k$ -- вектор-строка размера $P$, все элементы которого равны нулю, кроме $k$-ого элемента, равного единице. Пусть вектор $I'_{1\times (n+1)}$
   получен из вектора $I_{1\times (n+1)}$ обнулением $(n+1)$-го элемента. Тогда, индекс первой группы нулевых переменных равен
   $$ind_1=I'_{1\times (n+1)}\otimes (e^1+e^P)\otimes I_{1\times m}.$$
   Рассмотрим вторую группу нулевых переменных
   \begin{equation}
   \sum_{v=1}^m\sum_{(n+1,j)\notin E}(x_{v2j}+x_{v,P-1,j})=0.
   \end{equation}
   Пусть вектор $a$ -- $(n+1)$-ая строка матрицы смежности узлов $A$.
   Тогда, индекс нулевых переменных равен
   \begin{equation}
   ind_2=\bar a\otimes (e^2+e^{P-1})\otimes I_{1\times m},
   \end{equation}
   где $\bar a$ -- логическое отрицание вектора $a$.
       
   Индекс всех нулевых элементов вектора $\bar x$ равен $ind=ind_1\vee ind_2$.
   Если требуется чтобы все буровые установки были задействованы, то
   $ind=ind_1\vee ind_2\vee ind_3$, где
   $$ind_3=e^{n+1}\otimes e^2\otimes I_{1\times m}.$$

\end{enumerate}
Получаем следующую постановку транспортной задачи в форме QUBO
\begin{equation}
\frac12\bar xQ\bar x + v^T\bar x +const \to \min,
\end{equation}
где
\begingroup
\addtolength{\jot}{1em}
\begin{align*}
Q= & C\otimes D\otimes E_m +\alpha_1 E'_{n+1}\otimes I_{mP\times mP} + \alpha_2 I_{(n+1)\times (n+1)}\otimes E_{mP} +\ldots \\
 & + \alpha_3 \bar A\otimes D\otimes E_m + \alpha_4 G\otimes E_m - \alpha_5 E''_{n+1}\otimes H\otimes E_m, \\
\vspace{12pt}
v = &- \alpha_1 ( I'_{(n+1)\times 1}\otimes I_{mP\times 1}) - \alpha_2  I_{(n+1)mP\times 1},\\
const& = \frac{\alpha_1}{2}n + \frac{\alpha_2}{2}mP + \frac{\alpha_5}{2}m.
\end{align*}
\endgroup
При дополнительном условии равенства нулю части переменных $\bar x(ind)=0$, где
$$
 ind=ind_1\vee ind_2\vee ind_3.
$$

Для иллюстрации работы алгоритма на рисунках (\ref{fig1}) -- (\ref{fig3}) представлены найденные алгоритмом решения двух синтетических и одной реальной задачи. 
\begin{figure}[!htp]
\begin{center}
\includegraphics[scale=0.4]{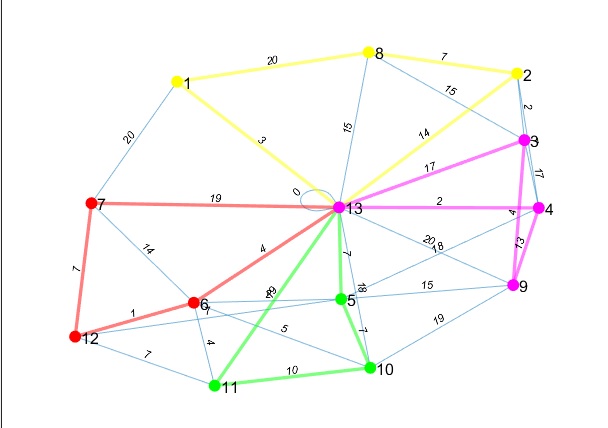}\caption{12 кустов и 4 буровые установки}\label{fig1}
\end{center}
\end{figure}
\begin{figure}[!htp]
\begin{center}
\includegraphics[scale=0.4]{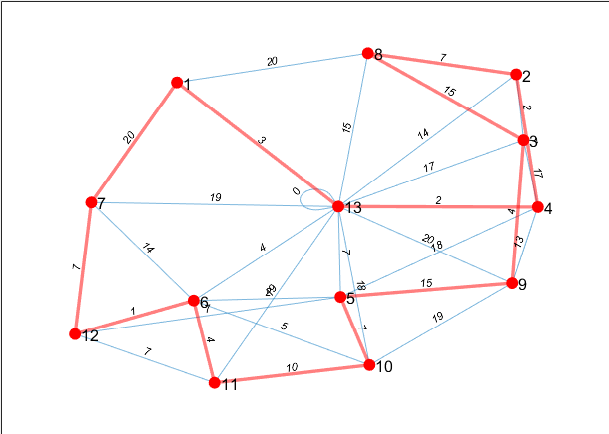}\caption{12 кустов и 1 буровая установка}\label{fig2}
\end{center}
\end{figure}
\begin{figure}[!htp]
\begin{center}
\includegraphics[scale=0.4]{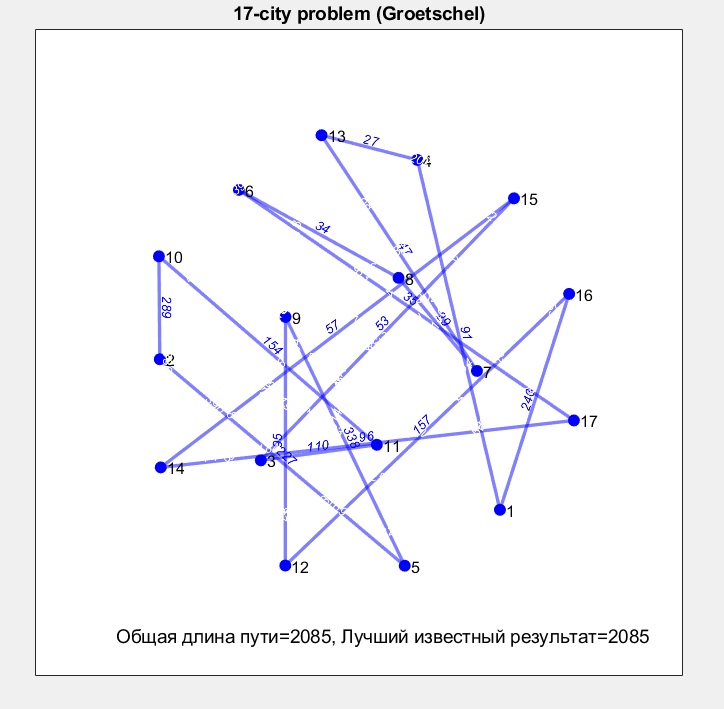}\caption{17-city problem (Groetschel)}\label{fig3}
\end{center}
\end{figure}

\subsection{Реферирование текста (McDonald Model Summarization)}

Пусть, предварительно обработанный текст, состоит из $n$ предложений. Длина каждого предложения $c_i$.
Коэффициент связи каждого предложения с текстом в целом, вычисленный по заданному алгоритму, равен $r_i$.
Коэффициент связи между различными предложениями, вычисленный по заданному алгоритму, равен $s_{ij}, i\ne j$ и $s_{ii}=0$.
Обозначим через $C$ вектор-строку с элементами $c_i$, через $R$ -- вектор-строку с элементами $r_i$ и через $S$ -- симметричную матрицу с элементами $s_{ij}$.
Пусть $K$ -- заданная максимальная длина сжатого текста. Определим бинарные переменные $x_i$ следующим образом:
если сжатый текст содержит $i$-ое предложение, то $x_i=1$, в противном случае $x_i=0$.

Постановка задачи:
\begin{equation}
-\sum_{i=1}^nr_ix_i+\frac{\alpha}{2}\sum_{i\ne j}s_{ij}x_ix_j \to \min, \alpha >0
\end{equation}
при условии, что
\begin{equation}
\sum_{i=1}^nc_ix_i\leq K.
\end{equation}
Для перевода неравенства в равенство, вычислим количество вспомогательных бинарных переменных. Напомним, что $p=\left \lceil\log_2(K+1)\right \rceil$. Вектор-строка бинаризации имеет вид 
$L=(1,2,\ldots,2^{p-2},K-2^{p-1}+1)$. Обозначим через $z$ вектор-столбец исходных и вспомогательных переменных $(x;y)$. 
По утверждению \ref{S6}, получаем следующую QUBO постановку задачи без учета константы 
\begin{equation}
\frac12 z^TQz+v^Tz\to \min,
\end{equation}
\begin{equation}
Q=\begin{pmatrix}
\alpha S+\rho C^TC     & \rho C^TL \\
\rho L^TC     & \rho L^TL 
\end{pmatrix},
\quad v=
\begin{pmatrix}
-R^T -\rho KC^T \\
-\rho KL^T  
\end{pmatrix},
\end{equation}
где $\rho >0$ -- штрафной коэффициент. 

\addcontentsline{toc}{section}{Заключение}
\section*{Заключение}
В работе получены полезные формулы для приведения исходной задачи к форме QUBO. 
Данные формулы существенно упрощают программную реализацию преобразования задачи к форме QUBO. 
В частности, позволяют ускорить формирование матрицы квадратичной формы. 
При размерностях задачи более $10000$ скорость поэлементного формирования матрицы на порядок ниже, чем скорость формирования с помощью произведения Кронекера.
В работе не рассматривались линейные ограничения с вещественными коэффициентами, и этот случай может стать предметом дальнейших исследований.

\end{widetext}

\begin{thebibliography}{0}%
\makeatletter
\providecommand \@ifxundefined [1]{%
 \@ifx{#1\undefined}
}%
\providecommand \@ifnum [1]{%
 \ifnum #1\expandafter \@firstoftwo
 \else \expandafter \@secondoftwo
 \fi
}%
\providecommand \@ifx [1]{%
 \ifx #1\expandafter \@firstoftwo
 \else \expandafter \@secondoftwo
 \fi
}%
\providecommand \natexlab [1]{#1}%
\providecommand \enquote  [1]{``#1''}%
\providecommand \bibnamefont  [1]{#1}%
\providecommand \bibfnamefont [1]{#1}%
\providecommand \citenamefont [1]{#1}%
\providecommand \href@noop [0]{\@secondoftwo}%
\providecommand \href [0]{\begingroup \@sanitize@url \@href}%
\providecommand \@href[1]{\@@startlink{#1}\@@href}%
\providecommand \@@href[1]{\endgroup#1\@@endlink}%
\providecommand \@sanitize@url [0]{\catcode `\\12\catcode `\$12\catcode
  `\&12\catcode `\#12\catcode `\^12\catcode `\_12\catcode `\%12\relax}%
\providecommand \@@startlink[1]{}%
\providecommand \@@endlink[0]{}%
\providecommand \url  [0]{\begingroup\@sanitize@url \@url }%
\providecommand \@url [1]{\endgroup\@href {#1}{\urlprefix }}%
\providecommand \urlprefix  [0]{URL }%
\providecommand \Eprint [0]{\href }%
\providecommand \doibase [0]{https://doi.org/}%
\providecommand \selectlanguage [0]{\@gobble}%
\providecommand \bibinfo  [0]{\@secondoftwo}%
\providecommand \bibfield  [0]{\@secondoftwo}%
\providecommand \translation [1]{[#1]}%
\providecommand \BibitemOpen [0]{}%
\providecommand \bibitemStop [0]{}%
\providecommand \bibitemNoStop [0]{.\EOS\space}%
\providecommand \EOS [0]{\spacefactor3000\relax}%
\providecommand \BibitemShut  [1]{\csname bibitem#1\endcsname}%
\let\auto@bib@innerbib\@empty
%</preamble>
\end{thebibliography}%


\addcontentsline{toc}{section}{\bibname}

\begin{thebibliography}{99}
\bibitem{1}
Kochenberger, G.; Hao, J.K.; Glover, F.; Lewis, M.; Lü, Z.; Wang, H.; Wang, Y.
The unconstrained binary quadratic programming problem: A survey.
J. Comb. Optim. 2014, vol. 28(1), pages 58-81, July.
URL \href{https://doi.org/10.1007/s10878-014-9734-0}{doi.org}.

\bibitem{2}
Andrew Lucas
Ising formulations of many NP problems.
Front. Phys., 12 February 2014, Sec. Interdisciplinary Physics, Volume 2 - 2014. 
URL \href{https://doi.org/10.3389/fphy.2014.00005}{doi.org}.

\bibitem{3}
Glover, F., Kochenberger, G., Hennig, R. et al. 
Quantum bridge analytics I: a tutorial on formulating and using QUBO models. 
Ann Oper Res 314, 141–183 (2022). 
URL \href{https://doi.org/10.1007/s10479-022-04634-2}{doi.org}.

\bibitem{4}
Fedorov, A. K., Gisin, N., Beloussov, S. M., Lvovsky,A. I. 
Quantum computing at the quantum advantage threshold: a down-to-business review.
arxiv preprint 2022. 
URL \href{https://arxiv.org/abs/2203.17181}{arxiv.org}.

\bibitem{5}
Tiunov, E. S., Ulanov, A. E., Lvovsky, A. I. 
Annealing by simulating the coherent ising machine. 
Opt. Express 27, 10288–10295 (2019). 
URL \href{http://opg.optica.org/oe/abstract.cfm?URI=oe-27-7-10288}{opg.optica.org}.


\bibitem{6}
Farhi, E., Goldstone, J., Gutmann, S.,  Sipser, M. 
Quantum computation by adiabatic evolution.
arhiv preprint 2000. 
URL \href{https://arxiv.org/abs/quant-ph/0001106}{arxiv.org}.

\bibitem{7}
Das, A.,  Chakrabarti, B. K. 
Colloquium: Quantum annealing and analog quantum computation. 
Rev. Mod. Phys. 80, 1061–1081 (2008). 
URL \href{https://doi.org/10.1103/RevModPhys.80.1061}{doi.org}.

\bibitem{8}
Albash, T., Lidar, D.A. 
Adiabatic quantum computation. 
Rev. Mod. Phys. 90, 015002 (2018). 
URL \href{https://doi.org/10.1103/RevModPhys.90.015002}{doi.org}.

\bibitem{9}
Aleksey S. Boev, Sergey R. Usmanov, Alexander M. Semenov, Maria M. Ushakova, Gleb V. Salahov, Alena S. Mastiukova, Evgeniy O. Kiktenko, Aleksey K. Fedorov
Quantum-inspired optimization for wavelength assignment.
Front. Phys., 18 January 2023, Sec. Quantum Engineering and Technology, Volume 10 - 2022. 
URL \href{https://doi.org/10.3389/fphy.2022.1092065}{doi.org}.

\bibitem{10}
Joey McCollum, Thomas Krauss
QUBO formulations of the longest path problem.
Theoretical Computer Science, Vol. 863,2021,Pages 86-101.
URL \href{https://doi.org/10.1016/j.tcs.2021.02.021}{doi.org}.

\bibitem{11}
Christos Papalitsas, Theodore Andronikos, Konstantinos Giannakis,
Georgia Theocharopoulou and Sofia Fanarioti
A QUBO Model for the Traveling Salesman Problem with Time Windows.
Algorithms 2019, 12, 224.
URL \href{https://doi.org/10.3390/a12110224}{doi.org}.

\bibitem{12}
Mohammad Asghari, Amir M. Fathollahi-Fard, S. M. J. Mirzapour Al-e-hashem, Maxim A. Dulebenets  Transformation and Linearization Techniques in Optimization: A State-of-the-Art Survey. 
Mathematics 2022, 10(2), 283. 
URL \href{https://doi.org/10.3390/math10020283}{doi.org}.

\bibitem{13}
Alidaee B, Kochenberger G, Lewis K, Lewis M, Wang H
A new approach for modeling and solving set packing problems.
February 2008, European Journal of Operational Research 186(2):504-512.
URL \href{https://doi.org/10.1016/j.ejor.2006.12.068}{doi.org}.

\bibitem{14}
I. M. Bomze, M. Budinich, P. M. Pardalos and M. Pelillo  The Maximum Clique Problem, 
In: D.-Z. Du and P. M. Pardalos, Eds., Handbook of Combinatorial Optimization. 
Kluwer Academic Pubisher, Dordrecht, 1999, pp. 1-74.

\bibitem{15}
Lloyd, S.; Mohseni, M.; Rebentrost, P.
Quantum algorithms for supervised and unsupervised machine learning. 
arxiv preprint 2013.
URL \href{https://doi.org/10.48550/arXiv.1307.0411}{doi.org}.

\bibitem{16}
Jonas Lang, Sebastian Zielinski and Sebastian Feld
Strategic Portfolio Optimization Using Simulated, Digital, and Quantum Annealing.
Appl. Sci. 2022, 12(23), 12288; 
URL \href{https://doi.org/10.3390/app122312288}{doi.org}.

\bibitem{17}
Oylum Seker, Merve Bodur, Hamed Pouya
Routing and Wavelength Assignment with Protection: A Quadratic
Unconstrained Binary Optimization Approach.
arxiv preprint 2021.
URL \href{https://arxiv.org/abs/2008.11924}{arhiv.org}.

\bibitem{18}
Date, P., Arthur, D.; Pusey-Nazzaro, L. 
QUBO formulations for training machine learning models. 
Sci Rep 11, 10029 (2021). 
URL \href{https://doi.org/10.1038/s41598-021-89461-4}{doi.org}.

\bibitem{19}
Farhi, E.;  Neven, H. 
Classification with quantum neural networks on near term processors. 
arXiv preprint 2018.
URL \href{https://arxiv.org/abs/1802.06002}{arxiv.org}.

\bibitem{20}
Frank Phillipson and Harshil Singh Bhatia
Portfolio Optimisation Using the D-Wave Quantum Annealer. 
arxiv preprint 2020.
URL \href{https://arxiv.org/pdf/2012.01121.pdf}{arxiv.org}.

\bibitem{21}
Samuel Mugel, Carlos Kuchkovsky, Escolástico Sánchez, Samuel Fernández-Lorenzo, Jorge Luis-Hita, Enrique Lizaso, and Román Orús
Dynamic portfolio optimization with real datasets using quantum processors and quantum-inspired tensor networks.
Phys. Rev. Research 4, 013006 – Published 3 January 2022.
URL \href{https://doi.org/10.1103/PhysRevResearch.4.013006}{doi.org}.

\bibitem{22}
Rahman, M. T., Han, S., Tadayon, N., Valaee, S., 2019. 
Ising model formulation of outlier rejection, with application in wifi based positioning, 
In: ICASSP 2019-2019 IEEE International Conference
on Acoustics, Speech and Signal Processing (ICASSP). IEEE, pp. 4405–4409.
URL \href{https://doi.org/10.1109/ICASSP.2019.8683807}{doi.org}.

\bibitem{23}
Harwood Stuart, Gambella Claudio, Trenev Dimitar, Simonetto Andrea,
Bernal David, and Greenberg Donny. Formulating and Solving Routing
Problems on Quantum Computers // IEEE Transactions on Quantum Engineering. -- 2021. -- Vol. 2. -- P. 1-17.
URL  \href{https://doi.org/10.1109/TQE.2021.3049230}{doi.org}

\bibitem{24}
Matthias Jung, Sven O. Krumke, Christof Schroth, Elisabeth Lobe, and Wolfgang Mauerer: QCEDA: Using Quantum Computers for EDA,
URL \href{https://arxiv.org/pdf/2403.12998.pdf}{arxiv.org}

\end{thebibliography}
\end{document}